\documentclass[10pt,oneside]{article}
\usepackage{mathrsfs}
\usepackage{amsmath,amsfonts,amssymb,amsthm}
\usepackage{caption}
\usepackage{subfigure}
\usepackage{color}
\usepackage{graphicx}
\usepackage{subfigure}
\usepackage{float}
\usepackage{paralist}
\usepackage{enumerate}
\usepackage{dsfont}
\usepackage{bm}
\usepackage[colorlinks,
linkcolor=red,
citecolor=blue
]{hyperref}
\usepackage{authblk}
\usepackage[T1]{fontenc}
\allowdisplaybreaks

\hypersetup{colorlinks=false} 
\newtheorem{thm}{Theorem}[section]
\newtheorem{lem}[thm]{Lemma}
\newtheorem{prop}[thm]{Proposition}
\newtheorem{rem}{Remark}[section]
\def\bma#1\ema{{\allowdisplaybreaks\begin{split}#1\end{split}}}

\newcommand{\R}{\mathbb{R}}

\numberwithin{equation}{section}

\begin{document}
	\title{Enhanced dissipation and temporal decay in the Euler-Poisson-Navier-Stokes equations} 
	\author[a]{ Young-Pil Choi \thanks{E-mail: ypchoi@yonsei.ac.kr(Y.-P. Choi)}}
	\author[b]{ Houzhi Tang \thanks{E-mail: houzhitang@amss.ac.cn(H.-Z. Tang)}}
	\author[c]{Weiyuan Zou \thanks{E-mail: zwy@amss.ac.cn(W.-Y. Zou)}}
		\affil[a]{Department of Mathematics,	Yonsei University, Seoul 03722, Republic of Korea}
	
	\affil[b]{Academy of Mathematics and Systems Science, Chinese Academy of Sciences, Beijing 100190, P. R. China}
	\affil[c]{College of Mathematics and Physics,
		Beijing University of Chemical Technology, Beijing 100029, P. R. China} 
	\date{}
	\renewcommand*{\Affilfont}{\small\it}	
	\maketitle

\begin{abstract}
This paper investigates the global well-posedness and large-time behavior of solutions for a coupled fluid model in $\mathbb{R}^3$ consisting of the isothermal compressible Euler-Poisson system and incompressible Navier-Stokes equations coupled through the drag force. Notably, we exploit the dissipation effects inherent in the Poisson equation to achieve a faster decay of fluid density compared to velocities. This strategic utilization of dissipation, together with the influence of the electric field and the damping structure induced by the drag force, leads to a remarkable decay behavior: the fluid density converges to equilibrium at a rate of $(1+t)^{-11/4}$, significantly faster than the decay rates of velocity differences $(1+t)^{-7/4}$ and velocities themselves $(1+t)^{-3/4}$ in the $L^2$ norm. Furthermore, under the condition of vanishing coupled incompressible flow, we demonstrate an exponential decay to a constant state for the solution of the corresponding system, the damped Euler-Poisson system. 
\end{abstract}
	\textbf{2020 MR Subject Classification:}\ 35B40, 35B65, 76N10

\tableofcontents

%
%
%
%
%
%
\section{Introduction}
In this paper, we study the dynamics of a coupled fluid model, the so-called Euler-Poisson-Navier-Stokes (EP-NS) system, consisting of the isothermal compressible Euler-Poisson system and incompressible Navier-Stokes equations coupled through  the drag force:
\begin{equation}\label{Main1}
	\left\{\begin{array}{lllll}
		\displaystyle \partial_t\rho+\textrm{div}(\rho u)=0, \quad (x,t)\in\mathbb{R}^3\times\mathbb{R}_+, \\
		\displaystyle \partial_t(\rho u)+\mathrm{div}\big(\rho u\otimes u\big)
		+\nabla\rho=-\rho(u-v)-\rho \nabla U, \\
		\displaystyle -\Delta U=\rho-1, \\
		\displaystyle \partial_tv+v\cdot\nabla v+\nabla P=\Delta v+\rho(u-v),\\
		\displaystyle \text{div}v=0.
	\end{array}\right.
\end{equation}
Here $\rho=\rho(x,t)$ and $u=u(x,t)$ denote the density and velocity for the compressible Euler fluid flow, respectively, $U=U(x,t)$ is the electric potential, $v=v(x,t)$ and $P=P(x,t)$ represent the velocity and pressure for the incompressible Navier-Stokes fluid flow, respectively. 

We supply the EP-NS system \eqref{Main1} with the initial data and far-field states: 
\begin{equation*}
	(\rho,u,v)(x,0)=(\rho_{0},u_{0},v_{0})(x), \quad \lim_{|x|\rightarrow +\infty}(\rho,u,v)(x,t)=(1,0,0). 
\end{equation*}
respectively.  Moreover, we assume the neutrality and non-vacuum conditions for the initial density:
\begin{equation}\label{neutral}
	\int_{\mathbb{R}^3}(\rho_0-1)\,dx=0 \quad \mbox{and} \quad \rho_0(x)>0 \mbox{ for all $x \in \R^3$}.
\end{equation}

In recent years, the interactions between particles and fluids have gained significant attention due to their widespread applications across various fields, including biotechnology, medicine, and the study of complex fluid phenomena. These interactions play a crucial role in phenomena such as sedimentation, the compressibility of droplets in sprays, the formation of cooling tower plumes, and the combustion processes in diesel engines \cite{MR2226800,MR3396231,o1981collective, Williams1958Spray}. For a more comprehensive understanding of these phenomena, readers are directed to \cite{MR2601394,o1981collective} for additional insights into the physical background of multiphase flow models. Beyond their practical implications, the mathematical analysis of models capturing these interactions holds significant importance. By rigorously examining the underlying mathematical structures and dynamics, a deeper understanding of the behavior of real-world systems can be attained. Furthermore, the development of mathematical models facilitates the prediction and optimization of processes in diverse applications. 
 
\paragraph{Related Literature and Objective} In the current work, we investigate the global-in-time existence and uniqueness of classical solutions to the EP-NS system \eqref{Main1} and its large-time behaviors. In the last decade, there have been significant developments in the mathematical analysis of the kinetic-fluid or coupled fluid models. To put our study in the proper perspective, we shall recall a few of the references on the Cauchy problem for coupled fluid models. Choi and Jung \cite{MR4604417} studied the hydrodynamic limit from the kinetic-fluid model, the nonlinear Vlasov-Poisson-Fokker-Planck equation coupled with the incompressible Navier-Stokes system, towards the EP-NS system \eqref{Main1} in the periodic spatial domain employing the relative entropy argument, see \cite{CCK16} for the case without Poisson interactions. Specifically, the global-in-time existences of solutions to both equations, the kinetic-fluid and the EP-NS equations, are established, and thus the hydrodynamic limit holds for all times. The large-time behavior of solutions to \eqref{Main1} in the periodic domain showing the exponential decay rate of convergences towards the equilibrium is obtained in \cite{CJ21}. In that paper, the Poincar\'e inequality is crucially used, and thus it is not clear to extend that result to the study of large-time behavior of solutions in the whole space. On the other hand, in the absence of the Poisson interaction, \eqref{Main1} reduces to the Euler-Navier-Stokes system, and Choi et al \cite{CJK24} and Huang et al \cite{huang2024global} established the global well-posedness and optimal decay estimates in the whole space. In the case where the coupled flow exhibits compressibility, Choi \cite{MR3546341} investigated the global existence and exponential decay rates of classical solutions in the periodic domain. Later, Wu et al \cite{MR4175837} extended the analysis of large-time behavior of solutions to the whole space case, obtaining optimal algebraic decay rates. Meanwhile, Wu and Zou \cite{MR4703478} focused on space-time estimates for this model. Notably, the coupling of the two-phase flow through the drag force term $\rho(u-v)$ enables the smooth effect of the Navier-Stokes equations to propagate to the Euler system, ensuring global well-posedness. Remarkably, observations on the large-time behavior in three-dimensional space indicate that the decay properties of the Euler flow remain consistent with those of incompressible or compressible Navier-Stokes flow through the drag force.
 
The primary focus of this paper lies in addressing the Cauchy problem of system \eqref{Main1} in $\mathbb{R}^3$. We are motivated by two key objectives:

\begin{enumerate}[(i)]
\item {\bf Preventing Singularities:} The Euler-Poisson equations are known to develop singularities in finite time, irrespective of the smoothness of the initial data \cite{MR3564592,MR1400743,MR1416291,MR1855666,MR1918784}. An intriguing question arises: can we utilize the smooth effect arising from coupled Navier-Stokes flow to establish the existence of solutions for the system \eqref{Main1}?
\item {\bf Understanding Electric Field Influence:} Building upon previous works on the compressible Navier-Stokes-Poisson system \cite{MR2609958,MR2917409} and Euler-Poisson system \cite{MR1637856, GHZ17, IP13, LW14}, we aim to elucidate the influence of the electric field on coupled fluid models.
\end{enumerate}

To address these questions, we employ a combination of methodologies, including the Hodge decomposition and spectral analysis for the linearized system around the equilibrium state, energy methods, and low-high frequency decomposition for nonlinear equations. Our results reveal that, owing to the smooth effect of incompressible flow, the formation of singularities in the Euler-Poisson system can be prevented through the drag force. Furthermore, influenced by the electric field, the density of particles exhibits an algebraic decay rate of $(1+t)^{-11/4}$ in $L^2$-norm, contrasting with the velocity decay rate of $(1+t)^{-3/4}$, which aligns with the order of the nonlinear term $\|u\cdot\nabla u\|_{L^2}$. In order to clearly distinguish these differences in convergence rates, under additional assumption on the initial velocities, we provide a lower bound decay rate for the velocity, indicating the decay rate $(1+t)^{-3/4}$ is optimal. We also investigate the decay rates for higher-order spatial derivatives of solutions, demonstrating enhancements in decay rates. Interestingly, for the corresponding linear system, the density decays to a constant state exponentially, while velocities tend to zero algebraically. This disparity diverges from previous works \cite{MR2609958,MR2917409}, where the density decay was dominated by the linear part and occurred at the same rate as $\|\nabla u\|_{L^2}$, namely $(1+t)^{-5/4}$. Consequently, while it may not be feasible to establish a lower-bound decay rate for the density of the EP-NS system, such a bound still holds for velocities. We also would like to point out that the linearized Euler-Poisson system for irrotational flows has the Klein-Gordon structure \cite{MR1637856}. Thus, the density cannot have the exponential decay rate of convergence, nor can it have a faster decay rate than the velocity. Moreover, the derivative of solutions does not enhance the convergence rate. However, in our coupled model, the two fluids are coupled through the drag force, which imparts a linear damping effect on velocity, resulting in a completely different behavior of fluid density, as stated above. Furthermore, in scenarios where the coupled incompressible flow vanishes, we observe exponential decay rates for solutions of the Euler-Poisson system with damping, further highlighting the influence of coupled flow on the Euler-Poisson system.

\paragraph{Notation} Here we introduce several notations used throughout the paper. $L^p(\mathbb{R}^3)$ and $W^{k,p}(\mathbb{R}^3)$ denote the usual Lebesgue and Sobolev space on $\mathbb{R}^3$, with norms $\|\cdot\|_{L^p}$ and $\|\cdot\|_{W^{k,p}}$, respectively.  When $p=2$, we denote $W^{k,p}(\mathbb{R}^3)$ by $H^k(\mathbb{R}^3)$ with the norm $\|\cdot\|_{H^k}$, and set
$$
\|u\|_{H^k(\mathbb{R}^3)}=\|u\|_{H^k},\quad \|u\|_{L^p(\mathbb{R}^3)}=\|u\|_{L^p}.
$$
We denote by $C$ a generic positive constant which may vary in different estimates. The notation $f_1\lesssim f_2$ means that there exists a constant $C>0$ such that $f_1\leq C f_2$. The symbol $f_1\simeq f_2$ represents the functions $f_1$ and $f_2$ are equivalent, which means $f_1 \lesssim f_2$ and $f_2 \lesssim f_1$.  If $\underset{x\rightarrow 0}{\lim}~\frac{f(x)}{g(x)}=C$ , we denote $f(x)=O(g(x))$.  For an integer $k$,  the symbol $\nabla^k$ denotes the summation of all terms $D^\ell=\partial_{x_1}^{\ell_1}\partial_{x_2}^{\ell_2}\partial_{x_3}^{\ell_3}$ with the multi-index $\ell$ satisfying $|\ell|=\ell_1+\ell_2+\ell_3=k$. For a function $f$, $\|f\|_{X}$ denotes the norm of $f$ in the space $X$. The notation $\|(f,g)\|_{X}$ indicates $\|f\|_{X}+\|g\|_{X}$. The Fourier transform of $f$ is denoted by $\hat{f}$ or $\mathscr{F}[f]$ and is defined by
\begin{equation*}
\hat{f}(\xi) =\mathscr{F}[f](\xi)=(2\pi)^{-\frac{3}{2}}\int_{\mathbb{R}^3} f(x)e^{-ix\cdot\xi}\,dx,\quad \xi\in\mathbb{R}^3.
\end{equation*}
The operator $\Lambda^a$ stands for the pseudodifferential operator defined by
\begin{equation}\label{Lambda}
	\Lambda^a f=\mathscr{F}^{-1}\Big(|\xi|^a \hat{f}(\xi)\Big)\ \ \text{for}\ a\in \mathbb{R}.
\end{equation}
		Meanwhile, we define the homogenous Sobolev space $\dot{H}^a(\mathbb{R}^3)$ of all $f$ and $a\in \mathbb{R}$ satisfying
		\begin{align*}
			\|f\|_{\dot{H}^a}\triangleq\|\Lambda^{a}f\|_{L^2}=\big\||\xi|^a\hat{f}(\xi)\big\|_{L^2}<\infty.
		\end{align*}
 
%
%
%
%
%
%
 \paragraph{Outline of Paper} The remainder of the paper is organized as follows. In Section \ref{sec:main}, we present our main results concerning the global-in-time existence and uniqueness of classical solutions to the EP-NS system \eqref{Main1}, as well as the damped Euler-Poisson system near equilibrium and its large-time behavior estimates. Section \ref{Sec2} focuses on the global-in-time solvability of the EP-NS system. In Section \ref{Sec3}, we explore the study of the large-time behavior of solutions. Here, we begin with a spectral analysis for the linearized system and subsequently investigate the nonlinear decay estimates for the EP-NS system, particularly focusing on upper-bound decay rates. Section \ref{Sec5} presents the optimal decay estimates for the nonlinear system, providing the lower-bound estimates. Finally, in Section \ref{Sec6}, we analyze the global Cauchy problem for the Euler-Poisson system with linear damping.

%
%
%
%
%
%
\section{Main results}\label{sec:main}
To capture the dissipation structure of the system \eqref{Main1} more clearly, we introduce a new variable $n=\log \rho$ and its initial counterpart $n_0=\log \rho_0$.   With this transformation, the system \eqref{Main1} can be reformulated as follows:
\begin{equation}\label{Main2}
	\left\{
	\begin{aligned}
		&\partial_t n+u\cdot\nabla n +\text{div}u=0,\\
		&\partial_tu+u\cdot\nabla u+\nabla n+u-v=-\nabla U, \\
		&-\Delta U=e^n-1, \\
		&\partial_tv+v\cdot\nabla v+\nabla P=\Delta v+e^{n}(u-v),\\
		&\text{div} v=0,
	\end{aligned}
	\right.
\end{equation}
with the initial data and far-field states
\begin{equation}\label{in-data}
	(n,u,v)|_{t=0}=(n_0,u_0,v_0),\quad  \lim_{|x|\rightarrow +\infty}(n,u,v)(x,t)=(0,0,0).
\end{equation}
The neutrality condition \eqref{neutral} become
\begin{equation}\label{netura}
	\int_{\mathbb{R}^3}(e^{n_0}-1)\,dx=0.
\end{equation}
Throughout the paper, we assume the above neutrality condition holds.

Now we are in a position to state the main results:
	\begin{thm}\label{thm1}
	Assume that the initial data satisfy 
	\[
	(n_0,u_0,v_0)\in H^3\cap \dot{H}^{-1}(\mathbb{R}^3) \times H^3(\mathbb{R}^3) \times H^3(\mathbb{R}^3)\quad \mbox{and} \quad {\rm{div}}\,v_0=0. 
	\]
	There exists a small positive constant $\delta_0$ such that if
	\begin{equation*}
		\|(n_0,u_0,v_0)\|_{H^3}+\|n_0\|_{\dot{H}^{-1}}\leq \delta_0,
	\end{equation*}
	then the Cauchy problem \eqref{Main2}-\eqref{in-data} admits a unique global-in-time classical solution 
	\[
	(n,u,v) \in L^\infty(\R_+; H^3\cap \dot{H}^{-1}(\R^3)) \times L^\infty(\R_+; H^3(\R^3)) \times L^\infty(\R_+; H^3(\R^3)) 
	\]
	such that 
	\begin{equation}\label{main-est}
		\begin{aligned}
	 &\|(n, u, v)(t)\|_{H^3}^2+\|n(t)\|_{\dot{H}^{-1}}^2\\
	 &\quad +C \int_0^t\left(\| n(s)\|_{H^3}^2+\| \nabla v(s)\|_{H^3}^2+\| (u-v)(s) \|_{H^3}^2\right) ds \leq C\delta_0^2,
	\end{aligned}
	\end{equation} 
	for any $t\in\mathbb{R}_+$. Moreover, if $\|(u_0,v_0)\|_{L^1}\leq \delta_0$, then it holds that
	\begin{align}
		\|n(t)\|_{\dot{H}^{-1}}&\leq C\delta_0(1+t)^{-2}, \nonumber\\ 
		\|\nabla^kn(t)\|_{L^2}&\leq C\delta_0(1+t)^{-\frac{11}{4}-\frac{k}{2}}, &&k=0,1,2,\label{032701}\\
		\|\nabla^3n(t)\|_{L^2}&\leq C\delta_0(1+t)^{-\frac{15}{4}}, \nonumber\\
		\|\nabla^k(u,v)(t)\|_{L^2}&\leq C\delta_0(1+t)^{-\frac{3}{4}-\frac{k}{2}}, &&k=0,1,2,3,
		\label{032702}
		\end{align}
	and the velocity difference $u-v$ has a faster rate compared to themselves
	\begin{align}\label{040101}
\|\nabla^k(u-v)(t)\|_{L^2}\leq C\delta_0(1+t)^{-\frac{7}{4}-\frac{k}{2}},\quad k=0,1,
	\end{align}
where the constant $C>0$ only depends on the initial data.
\end{thm}
\begin{rem}
	By using classical Sobolev interpolation inequalities, we obtain the $L^\infty$ decay rates of the solution as follows:
\[
	\|n(t)\|_{L^\infty}\leq C(1+t)^{-\frac{7}{2}} \quad \mbox{and} \quad \|(u,v)(t)\|_{L^\infty}\leq C(1+t)^{-\frac{3}{2}},
\]
where $C>0$ is independent of $t$.
\end{rem}
\begin{rem}
During the derivation of estimates \eqref{032701} and \eqref{032702}, we observe that the decay rates of $\|\nabla^k n(t)\|_{L^2}$ for $k=0,1,2$ are governed by $\|\nabla^k (u\cdot\nabla u)\|_{L^2}$, indicating the significant impact of nonlinear terms arising from velocity on the density. 
\end{rem}

As our second main result, we establish that the lower bound decay rates of velocities coincide with the upper bound for certain initial data, indicating the optimality of the observed rates in Theorem \ref{thm1}. This result can be stated as follows:
	\begin{thm}\label{thm3}
	Assume that all the conditions in Theorem \ref{thm1} hold. If the Fourier transform $(\hat{u}_0(\xi),\hat{v}_0(\xi))$ of the initial velocities $(u_0,v_0)$ satisfies
	\begin{equation}\label{in-data-optimal}
		\inf_{ |\xi|<r_0}\Big\{\Big|\hat{v}_0(\xi)+\Big(I-\frac{\xi\xi^t}{|\xi|^2}\Big)\hat{u}_0(\xi)\Big|\Big\} \geq \alpha_0>0,
	\end{equation}
	where $\alpha_0$ denotes a positive constant and $r_0$ is sufficiently small,  then it holds for large-time that 
	\begin{equation*}
		c_*(1+t)^{-\frac{3}{4}-\frac{k}{2}}\leq \|\nabla^k(u,v)(t)\|_{L^2}\leq C\delta_0(1+t)^{-\frac{3}{4}-\frac{k}{2}},\quad k=0,1,2,3, 
	\end{equation*}
	and
\begin{align*}
\bar{c}_*(1+t)^{-\frac{7}{4}-\frac{k}{2}}\leq \|\nabla^k(u-v)(t)\|_{L^2}\leq C\delta_0(1+t)^{-\frac{7}{4}-\frac{k}{2}},\quad k=0,1,
\end{align*}
where $c_*$ and $\bar{c}_*$ are positive constants independent of time.
\end{thm}

\begin{rem}
Here we comment on the smallness condition on $r_0$. This condition is not required in the previous work by Schonbek \cite{MR0837929} on the temporal decay estimates for the incompressible Navier-Stokes equations. The reason for this difference lies in the behavior of the solutions: the solution of the linear incompressible Navier-Stokes equations resembles the heat kernel in Fourier space, whereas the solution of our model behaves like the heat kernel only when $r_0$ is small. For more details, see Section \ref{Sec3} below.
\end{rem}
 
\begin{rem}
If the initial data satisfies ${\emph{\text{curl}}}~u_0=0$ and $\underset{ |\xi|<r_0}{\inf}|\hat{v}_0(\xi)|\geq \alpha_0$, then the condition \eqref{in-data-optimal} holds, indicating that this assumption is not empty.
\end{rem}

\begin{rem}\label{rem_den}
Since the density of the linearized system decays with an exponential rate, while the density of the nonlinear equations decays with an algebraic rate due to the effect of the incompressible flow, it is not expected to obtain the lower bound of decay rates for density.
\end{rem}

As mentioned in Remark \ref{rem_den}, linear analysis suggests exponential decay of solutions to the equilibrium when there is no coupling with incompressible flow. To further elucidate this observation, we consider the case where the coupling with incompressible flow is absent in \eqref{Main2}. In this case, the coupled system \eqref{Main2} reduces to the damped Euler-Poisson system:
\begin{equation}\label{M2}
	\left\{
	\begin{aligned}
		&\partial_t n+u\cdot\nabla n +\text{div}u=0,\\
		&\partial_tu+u\cdot\nabla u+\nabla n+u=-\nabla U, \\
		&-\Delta U=e^n-1,
	\end{aligned}
	\right.
\end{equation}
with the initial data and far-field states
\begin{equation}\label{i2}
	(n,u)|_{t=0}=(n_0,u_0) \quad \mbox{and} \quad \lim_{|x|\rightarrow +\infty}(n,u)(x,t)=(0,0).
\end{equation}

Then we state our third result on the global-in-time well-posedness and large-time behavior for the damped Euler-Poisson system \eqref{M2}. 
	\begin{thm}\label{Th3}
	Assume that the initial data $(n_0,u_0)\in H^3(\mathbb{R}^3)$ and  $n_0\in \dot{H}^{-1}(\mathbb{R}^3)$ satisfying 
	\begin{equation*}
		\|(n_0,u_0)\|_{H^3}+\|n_0\|_{\dot{H}^{-1}}\leq \delta_0,
	\end{equation*}
	where  $\delta_0>0$ is a small constant. Then the Cauchy problem \eqref{M2}-\eqref{i2} admits a unique global-in-time classical solution $(n,u) \in L^\infty(\R_+; H^3\cap \dot{H}^{-1}(\R^3)) \times L^\infty(\R_+; H^3(\R^3))$ such that 
	\begin{equation*}
		\begin{aligned}
		&\|(n, u)(t)\|_{H^3}^2+\|n(t)\|_{\dot{H}^{-1}}^2 +C \int_0^t\left(\| (n, u)(s)\|_{H^3}^2+\|n(s)\|_{\dot{H}^{-1}}^2 \right) ds\leq C\delta_0^2,
		\end{aligned}
	\end{equation*} 
	for any $t\in\mathbb{R}_+$. Moreover, it holds 
	\begin{equation*}
	\|(n, u)(t)\|_{H^3}+\|n(t)\|_{\dot{H}^{-1}} \leq C\delta_0e^{-C_*t},
	\end{equation*}
	where the positive constants $C$ and $C_*$ are independent of time. 
\end{thm}

\begin{rem}
	If the damping term in $\eqref{M2}_2$ is excluded, Guo \cite{MR1637856} constructed the global smooth irrotational flows with small velocity for the electron fluid. Additionally, the $L^\infty$ decay rate of solutions is derived as $(1+t)^{-p}$ for $p\in (1,\frac{3}{2})$. 
\end{rem}

%
%
%
%
%
%
 \paragraph{Idea of Proof} Now we outline the key ideas employed in the proof of the main theorems. Assuming the neutrality condition \eqref{netura} holds, we first establish global well-posedness using an energy method combined with {\it hypocoercivity-type} estimates. Note that the hypocoercivity-type estimates give the dissipation rate for the density. To investigate the large-time behavior of solutions, as stated in Introduction, we linearize the system \eqref{Main2} and estimate the Green function of that using Hodge decomposition and spectral analysis. This linear analysis reveals that the density decays to a constant state exponentially, while the velocities tend to zero at an algebraic rate. For the nonlinear system, we express the density $n(x,t)$ using Duhamel's principle as:
\begin{align}
    n(t) = S_{11} * n_0 + S_{12} * u_0 
    + \int_0^t S_{11}(t-s) * f_1(s) \,ds + \int_0^t S_{12}(t-s) * f_2(s) \,ds, \label{040601}
\end{align}
where $f_1=-u\cdot\nabla n$, $f_2=-u\cdot\nabla u-\nabla(-\Delta)^{-1}\left(e^n-1-n\right)$ and $(S_{ij})$ are from the Green function for the linearized system, see Section \ref{sssec_green} for details. Notably, the linear part of $n(t)$ decays faster than the nonlinear part, indicating that the decay rate of density is dominated by that of nonlinear term $u\cdot\nabla u$. Since the decay rate of the incompressible flow is the same as that of the heat equation, we infer that $\|u\|_{L^2}$ behaves similarly to the heat equation influenced by the velocity term $v$ in the momentum equations for $u$ in \eqref{Main2}. Building upon this observation, we define a suitable function $M(t)$ given by 
\[
\begin{aligned}
M(t)&\triangleq \sup _{0 \leq s \leq t} \Big\{\sum_{k=0}^2(1+s)^{\frac{11}{4}+\frac{k}{2}}\left\|\nabla^{k}n(s)\right\|_{L^2}
+(1+s)^{\frac{15}{4}}\|\nabla^3n(s)\|_{L^2}\\
&\hspace{7cm}+\sum_{k=0}^3(1+s)^{\frac{3}{4}+\frac{k}{2}}\left\|\nabla^{k}(u, v)(s)\right\|_{L^2}\Big\}.
\end{aligned}
\]
Our subsequent task is to prove that $M(t)$ is uniformly bounded in time, thereby establishing the desired decay estimates of solutions.

By estimating the Green function and by the definition of $M(t)$, we initially obtain:
\begin{equation*}
    \|\nabla ^k n(t)\|_{L^2} \leq Ce^{-R t} \delta_0 + C(1+t)^{-\frac{11}{4}-\frac{k}{2}} M(t)^2, \quad k=0,1,2,
\end{equation*}
where $\delta_0$ is appeared in Theorem \ref{thm1} and $R>0$ is independent of $t$. 
However, it is evident from \eqref{040601} that it cannot sustain third derivatives, exceeding the setting in the definition of $M(t)$. To address this issue, we introduce a low-high frequency decomposition:
\[
n= n^\ell+ n^h,
\]
where $n^\ell$ denotes the low-frequency part of solutions and $n^h$ represents the high-frequency part. Due to the smoothing effect of low frequencies, we have:
\begin{equation*} 
    \left\|\nabla^3 n^\ell(t)\right\|_2 \leq Ce^{-R t} \delta_0 + C(1+t)^{-\frac{15}{4}} M(t)^2.
\end{equation*}	
For the high-frequency part, we apply energy methods to the high-frequency part of the nonlinear system, yielding:
\begin{equation*}
    \|\nabla^{3}(n^{h},u^{h},v^{h})(t)\|_{L^2}
    \leq C (1+t)^{-\frac{15}{4}}\big(\delta_{0}+M(t)^{\frac{3}{2}}+M(t)^2\big).
\end{equation*}
Combining these results, we obtain:
\begin{align*}
    &\|\nabla ^k n(t)\|_{L^2} \leq  C(1+t)^{-\frac{11}{4}-\frac{k}{2}} (\delta_0+M(t)^2), \quad k=0,1,2, \quad \text{and}\\
     &\|\nabla ^3 n(t)\|_{L^2} \leq C(1+t)^{-\frac{15}{4}}(\delta_0+ M(t)^{\frac{3}{2}}+M(t)^2).
\end{align*}
The velocities $u,v$ are addressed similarly. Combining these estimates with the smallness of $\delta_0$, we establish the uniform bound of $M(t)$, which yields the upper bound decay rates of solutions. Moreover, by analyzing the equations for $u-v$, we obtain the upper bound decay rate of $\|(u-v)(t)\|_{L^2}$, faster than the decay rates of velocities themselves. 
 
In order to show the optimality of the decay rate of solutions, by selecting specific initial data (non-empty), we also obtain the estimates of lower bounds for $\|u(t)\|_{L^2}$ and upper bounds for $\|\Lambda^{-1}u(t)\|_{L^2}$. Then, by Sobolev interpolation inequalities, we establish the optimal decay rates of $u(t)$ for any order derivatives, similar to $\|\nabla^k v(t)\|_{L^2}$ and $\|(u-v)(t)\|_{L^2}$. Thanks to the key observation that the Green function can borrow a derivative due to the opposite signs of the drag force, we overcome the difficulty arising from the non-conservative structure of \eqref{Main1} to obtain the lower bound rate of $\|(u-v)(t)\|_{L^2}$. To elucidate the effect of coupled flow on the Euler-Poisson system, we consider the case where the coupled flow vanishes. We provide that the solution decays to a constant state at an exponential rate.

%
%
%
%
%
%
\section{Global well-posedness for the EP-NS system}\label{Sec2}	
In this section, we discuss the global-in-time well-posedness for the reformulated EP-NS system \eqref{Main2} providing detailed proof for the existence part of Theorem \ref{thm1}. 

%
%
%
%
%
%
\subsection{Local well-posedness}
		
We first establish the local-in-time existence theorem of the classical solution of the system \eqref{Main2}, which can be proved similarly to that in \cite{MR4604417} by using the contraction mapping principle. Here, we directly present the main result and omit the details of the proof.
\begin{thm}[Local-in-time well-posedness]\label{lem-loc}
	Assume that the initial data satisfy $(n_0,u_0,v_0)\in H^3(\mathbb{R}^3)$ and ${\rm{div}}v_0=0$, then there exists a short time $T_0>0$ such that the reformulated system \eqref{Main2} admits a unique classical solution $(n,u,v)$ satisfying
	\begin{equation*}
		\begin{aligned}
			&n\in C([0,T_0];H^3(\mathbb{R}^3))\cap C^1([0,T_0];H^2(\mathbb{R}^3)),\\
			&u\in C([0,T_0];H^3(\mathbb{R}^3))\cap C^1([0,T_0];H^2(\mathbb{R}^3)),~\nabla u\in L^2([0,T_0];H^2(\mathbb{R}^3)),\\
			&v\in C([0,T_0];H^3(\mathbb{R}^3))\cap C^1([0,T_0];H^1(\mathbb{R}^3)),~\nabla v\in L^2([0,T_0];H^3(\mathbb{R}^3)).
		\end{aligned}
	\end{equation*}
\end{thm}

\subsection{A priori estimates}\label{ssec_apriori}

To extend the local solution to a global one, the crucial step is to obtain uniform estimates. For later use, we introduce the following communicator estimates.
\begin{lem}{\rm (\cite[Lemma A.3]{MR2917409})}\label{lema2}
	Let $m\geq 1$ be an integer and define the communicator 
	\begin{equation*}
		[\nabla^m,f]g=\nabla^{m}(fg)-f\nabla^m g.
	\end{equation*}
	Then we have
	\begin{equation*}
		\big\|[\nabla^m,f]g\big\|_{L^p}\lesssim \|\nabla f\|_{L^{p_1}}\|\nabla^{m-1}g\|_{L^{p_2}}+
		\|\nabla^m f\|_{L^{p_3}}\|g\|_{L^{p_4}},
	\end{equation*}
	where $p, p_2, p_3\in(1,+\infty)$ and 
	\begin{equation*}
		\frac{1}{p}=\frac{1}{p_1}+\frac{1}{p_2}=\frac{1}{p_3}+\frac{1}{p_4}.
	\end{equation*}
\end{lem}

We then define energy $\mathcal{E}$ and corresponding dissipation rate $\mathcal{D}$ as
\begin{equation}\label{EDdefinition}
	\mathcal{E}(t)=\|(n, u, v)(t)\|_{H^3}+\|n(t)\|_{\dot{H}^{-1}} \quad \mbox{and} \quad \mathcal{D}(t)=\|n(t)\|_{H^3}+\|\nabla v(t)\|_{H^3}+\|(u-v)(t)\|_{H^3},
\end{equation}
respectively.  To close the nonlinear estimates, we provide the \emph{a priori} assumption that 
\begin{equation}\label{030802}
\mathcal{E}(t)=\|(n, v, v)(t)\|_{H^3}+\|n(t)\|_{\dot{H}^{-1}} \leq \delta,
\end{equation}
 with $\delta$ a sufficiently small positive constant.				
	\begin{prop}\label{p1}
Let $T>0$ and $(n,u,v)$ be the classical solution of system \eqref{Main2} satisfying the a priori assumption \eqref{030802}. Then, for $0<t\leq T$, it holds that
$$
\frac{d}{d t}\left\{\|(n, u, v)\|_{L^2}^2+\left\|\nabla  U\right\|_{L^2}^2\right\}+C_0\left(\|u-v\|_{L^2}^2+\|\nabla v\|_{L^2}^2\right) \leq 0,
$$
where $C_0$ is a positive constant independent of time.
\end{prop}
\begin{proof}
	Taking basic $L^2$ energy estimates on the system \eqref{Main1} gives
	$$
		\frac{d}{d t}\left\{\int_{\mathbb{R}^3}\Big(\frac{1}{2} \rho |u|^2+\rho \log \rho -\rho+1+\frac{1}{2}|v|^2+\frac{1}{2}|\nabla_x U|^2 \Big)\,dx\right\}+\int_{\mathbb{R}^3} \rho|u-v|^2 \,dx+\|\nabla v\|_{L^2}^2=0.
$$
Using \eqref{030802} and $n=\log \rho$, we obtain that for $\delta > 0$ small enough
	\begin{align*}
\rho\leq e^{\| n \|_{L^\infty}} \leq e^{C\|\nabla n\|_{H^1}} \leq e^{C\delta} \leq \frac{5}{4},\quad \rho\geq e^{-\|n\|_{L^\infty} } \geq e^{-C\|\nabla n\|_{H^1}} \geq e^{-C\delta} \geq \frac{4}{5},
	\end{align*}
	and
	$$
	\rho \log \rho-\rho+1 \simeq \frac{1}{2}(\rho-1)^2\simeq \frac{1}{2}n^2.
	$$
	There exists a positive constant $C_0$ such that 
	\begin{equation*}
		\frac{d}{d t}\left\{\|(n, u, v)\|_{L^2}^2+\left\|\nabla  U\right\|_{L^2}^2\right\}+C_0\left(\|u-v\|_{L^2}^2+\|\nabla v\|_{L^2}^2\right) \leq 0.
	\end{equation*}	
	This completes the proof.
\end{proof}

For the estimates of high-order derivatives, we rewrite the system \eqref{Main2} as
\begin{equation}\label{main4}
	\left\{
	\begin{aligned}
		&\partial_t n +{\rm{div}}\,u=f_1, \\
		&\partial_tu+\nabla n+u-v+\nabla (-\Delta )^{-1}n=f_2,\\
		&\partial_tv+v-\mathbb{P}u-\Delta v=f_3,
	\end{aligned}
	\right.
\end{equation}
where $\mathbb{P}$ is the Leary operator $\mathbb{P}= I+\nabla(-\Delta)^{-1}\text{div}$, and the nonlinear terms $f_1$, $f_2$, and $f_3$ are given by
\begin{equation}\label{nonlinearterms}
	\begin{aligned}
		& f_1=-u\cdot\nabla n,\\
		&f_2=-u\cdot\nabla u-\nabla(-\Delta)^{-1}\left(e^n-1-n\right), \quad \mbox{and}\\
		&f_3=-\mathbb{P}(v\cdot\nabla v)+\mathbb{P}((e^{n}-1)(u-v)),
	\end{aligned}
 \end{equation}
 respectively.

\begin{prop}\label{p2}
Let $T>0$ and $(n,u,v)$ be the classical solution of system \eqref{Main2} satisfying the a priori assumption \eqref{030802}. Then, for $0<t\leq T$ and  $k=0,1,2$, it holds that
\begin{equation*}
	\begin{aligned}
		& \frac{1}{2} \frac{d}{d t}\left(\left\|\nabla^k n\right\|_{H^1}^2+\left\|\nabla^{k+1} u\right\|_{L^2}^2+\left\|\nabla^{k+1} v\right\|_{L^2}^2\right)+\left\|\nabla^{k+2} v\right\|_{L^2}^2+\left\|\nabla^{k+1}(u-v)\right\|_{L^2}^2  \leq C \mathcal{E}(t) \mathcal{D}^2(t),
	\end{aligned}
\end{equation*}
where the positive constant $C$ is independent of time.
	\end{prop}
\begin{proof}
	Rewrite the equation $\eqref{Main2}_4$ as 
\[
		\partial_tv+v\cdot\nabla v+\nabla P-u+v=\Delta v+(\rho-1)(u-v).
\]
This together with \eqref{main4} yields		
	\begin{equation}\label{030805}
		\begin{aligned}
			& \frac{1}{2} \frac{d}{d t}\left(\|\nabla n\|_{L^2}^2+\|\nabla u\|_{L^2}^2+\|\nabla v\|_{L^2}^2\right)+\left\|\nabla^2 v\right\|_{L^2}^2+\|\nabla(u-v)\|_{L^2}^2\\
			&\quad =-\int_{\mathbb{R}^3} \nabla \nabla(-\Delta)^{-1} n \cdot  \nabla u \,dx+\int_{\mathbb{R}^3} \nabla f_1 \cdot \nabla n \,dx+\int_{\mathbb{R}^3}\nabla f_2 \cdot \nabla u \,dx\\
			&\qquad -\int_{\mathbb{R}^3} \nabla(v \cdot \nabla v) \cdot \nabla v \,dx+\int_{\mathbb{R}^3} \nabla((\rho-1)(u-v)) \cdot \nabla v \,dx.
		\end{aligned}
	\end{equation}
	In terms of integration by parts and the equation $\eqref{main4}_1$, we have
	\begin{equation}\label{030803}
	\begin{aligned}
		 -\int_{\mathbb{R}^3} \nabla \nabla(-\Delta)^{-1} n \cdot \nabla u \,dx 
		&=  -\int_{\mathbb{R}^3}(-\Delta)^{-1} n \cdot \Delta \text{div} u \,dx\\
	&=-\int_{\mathbb{R}^3}(-\Delta)^{-1} n \cdot \Delta\left(-\partial_tn+f_1\right) \,dx \\
			& =\int_{\mathbb{R}^3}(-\Delta)^{-1} n \cdot \Delta \partial_tn \,dx-\int_{\mathbb{R}^3}(-\Delta)^{-1} n \cdot \Delta f_1 \,dx \\
			& =-\int_{\mathbb{R}^3} n \cdot \partial_tn\,dx-\int_{\mathbb{R}^3}(-\Delta)^{-1} n \cdot \Delta f_1 \,dx \\
			& =-\frac{1}{2} \frac{d}{d t}\|n\|_{L^2}^2-\int_{\mathbb{R}^3}(-\Delta)^{-1} n \cdot \Delta f_1 \,dx.
		\end{aligned}
	\end{equation}
	The second term on the right-hand side of \eqref{030803} can be estimated as
	\begin{equation}\label{030804}
		\begin{aligned}
			\left|\int_{\mathbb{R}^3}(-\Delta)^{-1} n \cdot \Delta f_1 \,dx\right| & \leq\left|\int_{\mathbb{R}^3} \nabla(-\Delta)^{-1} n \cdot \nabla f_1 \,dx\right| \\
			& \leq\left|\int_{\mathbb{R}^3} \nabla(-\Delta)^{-1} n \cdot \nabla(u \cdot \nabla n) \,dx\right| \\
			& \leq\left|\int_{\mathbb{R}^3} n u \cdot \nabla n \,dx\right| \\
			& \leq\|n\|_{L^2}\|u\|_{L^\infty}\|\nabla n\|_{L^2} \\
			& \leq C\|n\|_{H^{1}}^2\|\nabla u\|_{H^1} \\
			& \leq C\mathcal{E}(t)\mathcal{D}^2(t).
		\end{aligned}
	\end{equation}
	Combining \eqref{030803} and \eqref{030804} together gives 
	\begin{equation*}\label{030806}
		-\int_{\mathbb{R}^3} \nabla \nabla(-\Delta)^{-1} n \cdot \nabla u \,dx \leq-\frac{1}{2} \frac{d}{d t}\| n \|_{L^2}^2+C \mathcal{E}(t)\mathcal{D}^2(t).
	\end{equation*}
For the second term on the right-hand side of \eqref{030805}, we get
	\begin{equation*}
		\begin{aligned}
			 \left|\int_{\mathbb{R}^3} \nabla f_1 \cdot \nabla n \,dx\right| &\leq \left| \int_{\mathbb{R}^3} \nabla(u \cdot \nabla n) \cdot \nabla n \,dx\right| \\
			& \leq\left|\int_{\mathbb{R}^3}(\nabla(u \cdot \nabla n)-u \cdot \nabla \nabla n) \cdot \nabla n \,dx\right|+\left|\int_{\mathbb{R}^3} u \cdot \nabla \nabla n \cdot \nabla n \,dx\right| \\
			& \leq C\left(\|\nabla u\|_{L^{\infty}}\|\nabla n\|_{L^2}+\|\nabla n\|_{L^\infty}\|\nabla u\|_{L^2}\right)\|\nabla n\|_{L^2}+C\|\text{div}u\|_{L^\infty}\|\nabla n\|_{L^2}^2 \\
			& \leq C \mathcal{E}(t) \mathcal{D}^2(t).
		\end{aligned}
	\end{equation*}
	The third term on the right-hand side of \eqref{030805} can be estimated as 	
	\begin{equation*}
		\begin{aligned}
			 \left|\int_{\mathbb{R}^3} \nabla f_2 \cdot \nabla u \,dx\right| &\leq\left|\int_{\mathbb{R}^3} \nabla(u \cdot \nabla u) \cdot \nabla u \,dx\right|+\left|\int_{\mathbb{R}^3} \nabla \left(\nabla(-\Delta)^{-1}\left(e^n-1-n\right)\right) \cdot \nabla u \,dx\right| \\
			& \leq \left| \int_{\mathbb{R}^3}(\nabla(u \cdot\nabla u)-u \cdot \nabla \nabla u) \cdot \nabla u \,dx\right|+\left| \int_{\mathbb{R}^3} u \cdot \nabla \nabla u \cdot \nabla u \,dx\right| \\
			&\quad+\left|  \int_{\mathbb{R}^3} \nabla(-\Delta)^{-1}\left(e^n-1-n\right) \cdot \Delta u \,dx  \right| \\
			& \leq C \| \nabla u \|_{L^\infty}\| \nabla  u\|_{L^2}^2+C(\|e^n-1-n\|_{L^2}+\|e^n -1-n\|_{L^1} ) \| \nabla^2 u \|_{L^2} \\
			& \leq C\mathcal{E}(t) \mathcal{D}^2(t)+C\left(\|n\|_{L^\infty}\|n\|_{L^2}+\|n\|_{L^2}^2\right)\left\| \nabla^2 u\right\|_2 \\
			& \leq C \mathcal{E}(t)\mathcal{D}^2(t).
		\end{aligned}
	\end{equation*}
Here we used the fact that
	$$
		\begin{aligned}
		 \left\|\nabla(-\Delta)^{-1}\left(e^n-1-n\right)\right\|_{L^2}&=\left\|\nabla K *\left(e^n-1-n\right)\right\|_{L^2}\\
			& \leq C(\left\|e^n-1-n\right\|_{L^2}+\left\|e^n-1-n\right\|_{L^1}) \\
			& \leq C\|n\|_{L^\infty}\|n\|_{L^2}+C\|n\|_{L^2}^2,
		\end{aligned}
	$$
	where $K$ is the fundamental solution of the Laplace equation.
	
	The remaining terms in \eqref{030805} can also be treated similarly
	\begin{equation*}
		\begin{aligned}
			&-\int_{\mathbb{R}^3} \nabla(v \cdot \nabla v) \cdot \nabla v \,dx
			+\int_{\mathbb{R}^3} \nabla((\rho-1)(u-v)) \cdot \nabla v \,dx\\
			 &\quad =-\int_{\mathbb{R}^3}(\nabla(v \cdot \nabla v)-v \cdot \nabla \nabla v)\cdot\nabla v \,dx-\int_{\mathbb{R}^3} v \cdot \nabla \nabla v \cdot \nabla v \,dx-\int_{\mathbb{R}^3} (\rho-1)(u-v) \cdot \Delta v \,dx  \\
			&\quad \leq C\|\nabla v\|_{L^\infty}\|\nabla v\|_{L^2}^2+\|\rho-1\|_{L^3}\|u-v\|_{L^6}\|\Delta v\|_{L^2}\\
			&\quad \leq C\|\nabla v\|_{L^\infty}\|\nabla v\|_{L^2}^2+C\|n\|_{H^1}\|\nabla(u-v)\|_{L^2}\|\nabla^2 v\|_{L^2}\\
			&\quad \leq C\mathcal{E}(t)\mathcal{D}^2(t).
		\end{aligned}
	\end{equation*}
	Thus, we conclude from the above fact that
	\begin{align}\label{030809}
			\frac{1}{2} \frac{d}{d t}\left(\|n\|_{H^1}^2+\|\nabla u\|_{L^2}^2+\|\nabla v\|_{L^2}^2\right)+\left\|\nabla^2 v\right\|_{L^2}^2+\|\nabla(u-v)\|_{L^2}^2 
			 \leq C\mathcal{E}(t)\mathcal{D}^2(t).
	\end{align}	
Similarly, we also prove for $k=1,2$ that  
	\begin{equation}\label{030808}
		\begin{aligned}
			& \frac{1}{2} \frac{d}{d t}\left(\left\|\nabla^k n\right\|_{H^1}^2+\left\|\nabla^{k+1} u\right\|_{L^2}^2+\left\|\nabla^{k+1} v\right\|_{L^2}^2\right)+\left\|\nabla^{k+2} v\right\|_{L^2}^2+\left\|\nabla^{k+1}(u-v)\right\|_{L^2}^2 \\
			& \quad \leq C \mathcal{E}(t) \mathcal{D}^2(t).
		\end{aligned}
	\end{equation}
Combining \eqref{030809} and \eqref{030808} concludes the desired result.
\end{proof}

Note that the dissipation rate $\mathcal{D}$ includes the density term, $\|n\|_{H^3}$. Thus, in order to close the a priori estimates and provide the uniform-in-time bound for $\mathcal{E}$, we need to have the dissipation rate for the density. To this end, we provide a hypocoercivity-type estimate.

	\begin{prop}\label{p3}
Let $T>0$ and $(n,u,v)$ be the classical solution of system \eqref{Main2} satisfying the a priori assumption \eqref{030802}. Then, for $0<t\leq T$, it holds that
\begin{equation}\label{030812}
\begin{aligned}
	& \frac{d}{d t}\left\{\sum_{k= 0}^2 \int_{\mathbb{R}^3} \nabla^{k} u \cdot \nabla^{k+1} n \,dx\right\}+\frac{1}{2}\|n\|_{H^3}^2   \leq C_4\left(\|\nabla u\|_{H^2}^2+\|u-v\|_{H^2}^2\right)+C \mathcal{E}(t)\mathcal{D}^2(t),
\end{aligned}
\end{equation}
where $C_4$ is a positive constant independent of time.
\end{prop}
\begin{proof}	
Multiplying $\eqref{main4}_2$ by $\nabla n$, then integrating the resulting equation over $\mathbb{R}^3$ yields 
		\begin{align*}
			& \frac{d}{d t} \int_{\mathbb{R}^3} u \cdot \nabla n \,dx+\|\nabla n\|_{L^2}^2+\int_{\mathbb{R}^3} \nabla(-\Delta)^{-1} n \cdot \nabla n \,dx \\
			&\quad =\int_{\mathbb{R}^3} u \cdot \nabla \partial_tn \,dx-\int_{\mathbb{R}^3}(u-v) \cdot \nabla n \,dx+\int_{\mathbb{R}^3}\left(-u \cdot \nabla u-\nabla(-\Delta)^{-1}\left(e^n-1-n\right)\right) \cdot \nabla n \,dx\\
			&\quad =\|\text{div} u\|_{L^2}^2+\int_{\mathbb{R}^3} u \cdot \nabla f_1 \,dx-\int_{\mathbb{R}^3}(u-v) \cdot \nabla n \,dx+\int_{\mathbb{R}^3}\left(-u \cdot \nabla u-\nabla(-\Delta)^{-1}\left(e^n-1-n\right)\right) \cdot \nabla n \,dx.
		\end{align*}
Note that
	\begin{equation*}
		\int_{\mathbb{R}^3} \nabla(-\Delta)^{-1} n \cdot \nabla n \,dx=-\int_{\mathbb{R}^3} \Delta(-\Delta)^{-1} n \cdot n \,dx=\|n\|_{L^2}^2.
	\end{equation*}
	and
	\begin{equation*}
		\int_{\mathbb{R}^3}(u-v) \cdot \nabla n \,dx \leq\frac{1}{2}\|u-v\|_{L^2}^2+ \frac{1}{2}\|\nabla n\|_{L^2}^2.
	\end{equation*}
	Applying a similar argument developed in Proposition \ref{p2} yields
	\begin{equation*}
		\int_{\mathbb{R}^3}\left(u \cdot \nabla f_1 -u \cdot \nabla u-\nabla(-\Delta)^{-1}\left(e^n-1-n\right)\right) \cdot \nabla n \,dx \leq C \mathcal{E}(t)\mathcal{D}^2(t).
	\end{equation*}
Thus, we have
	\begin{equation}\label{0308010}
		\begin{aligned}
			& \frac{d}{d t} \int_{\mathbb{R}^3} u \cdot \nabla n \,dx+\frac{1}{2}\|\nabla n \|_{L^2}^2+\|n\|_{L^2}^2  \leq C_2\left(\|\nabla  u\|_{L^2}^2+\|u-v\|_{L^2}^2\right)+C\mathcal{E}(t)\mathcal{D}^2(t).
		\end{aligned}
	\end{equation}
for some $C_2 > 0$.	
	Similarly, there exists a positive constant $C_3$ such that 
	\begin{equation}\label{0308011}
		\begin{aligned}
			& \frac{d}{d t} \int_{\mathbb{R}^3} \nabla^k u \cdot \nabla^{k+1}n \,dx+\frac{1}{2}\left\|\nabla^{k+1} n\right\|_{L^2}^2+\left\|\nabla^k n\right\|_{L^2}^2 \\
			& \quad \leq C_3\left(\left\|\nabla^{k+1} u\right\|_{L^2}^2+\left\|\nabla^k(u-v)\right\|_{L^2}^2\right)+C \mathcal{E}(t)\mathcal{D}^2(t),
		\end{aligned}
	\end{equation}
with $k=1,2$. 

Choosing $C_4=\max \left\{C_2, C_3\right\}$, we sum \eqref{0308011} with respect to $k$ from $1$ to $2$ and use \eqref{0308010} to finish the proof.
\end{proof}

%
%
%
%
%
%
\subsection{Proof of Theorem \ref{thm1}: Global-in-time well-posedness}
We now close the energy estimates which provide the global-in-time existence and uniqueness of classical solutions for the system \eqref{Main2}.

It follows from Section \ref{ssec_apriori} that
	\begin{equation}\label{030813}
		\frac{d}{d t}\Big\{\|(n, u, v)(t)\|_{H^3}^2+\|\nabla U(t)\|_{L^2}^2\Big\}+C_5\left(\|\nabla v\|_{H^3}^2+\|u-v\|_{H^3}^2\right) \leq C \mathcal{E}(t)\mathcal{D}^2(t),
	\end{equation}
for some $C_5>0$ independent of time. 

		Taking $\eqref{030813}$+$\eqref{030812} \times \beta_1$ with $\beta_1 > 0$ small enough gives
	\begin{equation}\label{030814}
		\begin{aligned}
			& \frac{d}{d t}\left\{\|(n, u, v)(t)\|_{H^3}^2+\|\nabla U(t)\|_{L^2}^2+\beta_1 \sum_{k=0}^2 \int_{\mathbb{R}^3} \nabla^{k} u \cdot \nabla^{k+1} n \,dx\right\} \\
			& \quad +\frac{1}{2} \beta_1\|n\|_{H^3}^2+C_5\left(\|\nabla v\|_{H^3}^2+\|u-v\|_{H^3}^2\right) \\
			& \qquad \leq C \mathcal{E}(t)\mathcal{D}^2(t)+C_4 \beta_1\left(\|\nabla u\|_{H^2}^2+\|u-v\|_{H^2}^2\right)+C\beta_1\mathcal{E}(t)\mathcal{D}^2(t).
		\end{aligned}
	\end{equation}
	Due to the smallness of $\beta_1$, we find that 
	$$
	\|(n, u, v)(t)\|_{H^3}^2+\|\nabla U(t)\|_{L^2}^2+\beta_1 \sum_{k=0}^2 \int_{\mathbb{R}^3} \nabla^k u \cdot \nabla^{k+1}n \,dx \simeq \|(n, u, v)(t)\|_{H^3}^2+\|\nabla U(t)\|_{L^2}^2.
	$$
Combining this with \eqref{030814} shows 
	$$
		 \frac{d}{d t}\Big\{\|(n, u, v)(t)\|_{H^3}^2+\|\nabla U(t)\|_{L^2}^2\Big\}+C_6\left(\|n\|_{H^3}^2+\|\nabla v\|_{H^3}^2+\|u-v\|_{H^3}^2\right) \leq C \mathcal{E}(t)\mathcal{D}^2(t),
	$$
	where $C_6>0$ is a constant independent of time.
	
	Integrating the above inequality with respect to time over $[0, t]$ and using the a priori assumption \eqref{030802} yields
	$$
	\begin{aligned}
		& \|(n, u, v)(t)\|_{H^3}^2+\|\nabla U(t)\|_{L^2}^2+C_6 \int_0^t\left(\| n(\tau)\|_{H^3}^2+\| \nabla v(\tau)\|_{H^3}^2+\| (u-v)(\tau) \|_{H^3}^2\right) d \tau \\
		&\quad \leq\left\|\left(n_0, u_0, v_0\right)\right\|_{H^3}^2+\|\nabla(-\Delta)^{-1}U_0\|_{L^2}^2+C \int_0^t \mathcal{E}(\tau)\mathcal{D}^2(\tau) \,d \tau \\
		&\quad \leq\left\|\left(n_0, u_0, v_0\right)\right\|_{H^3}^2+C\|n_0\|_{\dot{H}^{-1}}^2+C \int_0^t \mathcal{E}(\tau)\mathcal{D}^2(\tau)\, d \tau \\
		&\quad \leq C_7\delta_0^2+C \delta \int_0^t \mathcal{D}^2(\tau) \,d \tau,
	\end{aligned}
	$$
	where $C_7>0$ is independent of time.
	
On the other hand, by the definition of $\mathcal{E}(t)$ and $\mathcal{D}(t)$ in \eqref{EDdefinition}, we obtain
	$$
	\mathcal{E}^2(t)+C_{6} \int_0^t \mathcal{D}^2(\tau)\, d \tau \leq C_7\delta_0^2+C \delta \int_0^t \mathcal{D}^2(\tau) \,d \tau.
	$$
	Since $\delta$ is small, one has
	$$
	\mathcal{E}^2(t)+\frac{1}{2} C_6\int_0^t \mathcal{D}^2(\tau) \,d \tau \leq C_7\delta_0^2.
	$$

Choosing $\delta_0$ sufficiently small such that $\delta_0<\frac{1}{2\sqrt{C_7}}\delta$, we have $\mathcal{E}(t) \leq \frac{1}{2} \delta$, which closes the a priori assumption \eqref{030802}. Then, based on the local existence result, Theorem \ref{lem-loc}, together with the classical continuation argument, global well-posedness is established, along with \eqref{main-est}.

%
%
%
%
%
%
	\section{Large-time behavior for the EP-NS system}\label{Sec3}

In this section, we investigate the large-time behavior of the EP-NS system demonstrating the decay estimates of solutions as stated in Theorem \ref{thm1}. Specifically, we establish time-decay estimates employing spectral analysis via the Hodge decomposition.

We define operators $\mathcal{K}_{1}$ and $\mathcal{K}_{\infty}$ on $L^{2}$ as
\begin{equation*}
	\mathcal{K}_{1}f=f^\ell=\mathscr{F}^{- 1}\big(\hat\chi_{1}(\xi)\mathscr{F}[f](\xi)\big) \quad \mbox{and} \quad	\mathcal{K}_{\infty}f=f^h=\mathscr{F}^{- 1}\big(\hat\chi_{\infty}(\xi)\mathscr{F}[f](\xi)\big),
\end{equation*}
respectively, where $\hat{\chi}_{j}(\xi)(j=1,\infty)\in C^{\infty}(\mathds{R}^{3})$, $0\leq \hat\chi_{j}\leq 1$ are smooth cut-off functions defined by
\begin{equation*}
	\hat\chi_{1}(\xi)=\left\{
	\begin{array}{l}
		1\quad (|\xi|\leq r_{0}),\\
		0\quad (|\xi|\geq R_0)
	\end{array}
	\right.
	\quad \mbox{and} \quad \hat\chi_{\infty}(\xi)=1-\hat\chi_{1}(\xi).
\end{equation*}
Here $r_0 > 0$ is sufficiently small and $R_0 > 0$ is sufficiently large.

To analyze the large-time behavior of the solutions $\eta=(n,u,v)^t$ in frequency space, we adopt the low-high frequency decomposition for the solution: 
\begin{equation*}
	\eta(t,x)=\eta^{\ell}(t,x)+\eta^{h}(t,x)\triangleq(n^{\ell},u^{\ell},v^{\ell})+(n^{h},u^{h},v^{h}),
\end{equation*}
where $\eta^{\ell}(t,x)\triangleq\mathcal{K}_1\eta(t,x)$ denotes the low-frequency part and $\eta^{h}(t,x)\triangleq\mathcal{K}_{\infty}\eta(t,x)$ represents the high-frequency part.

Before proceeding further, we present an auxiliary lemma providing the Sobolev embedding for low- and high-frequency parts.
	\begin{lem}{\rm (\cite[Lemma 2.4]{huang2024global})}\label{lema6}
	Let $m,n \in \mathbb{R}_+$. For $0\leq n<m$, there exist positive a constant $C$ such that for $f\in H^m$,
	\begin{equation*}
		\|\nabla^{m} f^\ell\|_{L^2}\leq C\|\nabla^{n} f^\ell\|_{L^2},\quad \|\nabla^{n}f^h\|_{L^2}\leq C\|\nabla^{m}f^h\|_{L^2},
	\end{equation*}
	and 
	\begin{equation*}
		\|\nabla^{n} f^\ell\|_{L^2}\leq C\|\nabla^{n} f\|_{L^2},\quad\|\nabla^{n}f^h\|_{L^2}\leq C\|\nabla^{n}f\|_{L^2}.
	\end{equation*}
\end{lem}

%
%
%
%
%
%
	\subsection{Spectral analysis for the linearized EP-NS system}\label{ssec_lina}
	Letting $f_i \equiv0$, $i=1,2,3$ in \eqref{main4}, we have the following linearized system:
\begin{equation}\label{Linear}
	\left\{\begin{aligned}
		&\partial_t n+\operatorname{div} u=0, \\
		&\partial_t u+\nabla n+u-v+\nabla(-\Delta)^{-1} n=0, \\
		&\partial_t v+v-\mathbb{P} u-\Delta v=0.
	\end{aligned}\right.
\end{equation}	
To estimate the optimal decay rate of convergences of solutions to the above system, we make use of spectral analysis. To this end, we first decompose the velocities $u$ and $v$ using the Hodge decomposition, separating them into their divergence and curl components. Specifically, let us define
$$
q=\Lambda^{-1} \text{div} u, \quad w=\Lambda^{-1} \text{curl}  u, \quad \mbox{and} \quad b=\Lambda^{-1} \text{curl} v ,
$$
where  the operator $\Lambda^{-1}$ is the pseudodifferential operator defined as \eqref{Lambda}. 

We then derive
\begin{equation}\label{A}
	\left\{\begin{aligned}
		&\partial_t n+\Lambda q=0, \\
		&\partial_t q+q-\Lambda n-\Lambda^{-1} n=0,
	\end{aligned}\right.
\end{equation}
and
\begin{equation}\label{B}
\left\{\begin{aligned}
		&\partial_t w+w-b=0, \\
		&\partial_t b+b-w+\Lambda^2 b=0.
	\end{aligned}\right.
\end{equation}
Taking Fourier transform of \eqref{A} with respect to $x$ yields
$$
\left\{\begin{aligned}
	&\partial_t \hat{n}+|\xi| \hat{q}=0, \\
	&\partial_t \hat{q}+\hat{q}-|\xi| \hat{n}-|\xi|^{-1} \hat{n}=0,\\
	&(\hat{n}, \hat{q})|_{t=0}=(\hat{n}_0, \hat{q}_0).
\end{aligned}\right.
$$			
Letting $\hat{\Phi}=(\hat{n}, \hat{q})^t$, we can rewrite the above system as
$$
\left\{\begin{aligned}
	&\partial_t \hat{\Phi}+\hat{S}_1 \hat{\Phi}=0, \\
	&\hat{\Phi}|_{t=0}=\hat{\Phi}_0,
\end{aligned}\right. \quad \mbox{with}\quad \hat{S}_1=\left(\begin{array}{cc}
	0 & |\xi| \\
	-|\xi|-|\xi|^{-1} & 1
\end{array}\right).
$$
The eigenvalues of $\hat{S}_1$ are computed by solving
\begin{equation*}
	\begin{aligned}
		\text{det}(\lambda I+\hat{S}_1)=\left|\begin{array}{cc}
			\lambda & |\xi| \\
			-|\xi|-|\xi|^{-1} & \lambda+1
		\end{array}\right| & =\lambda^2+\lambda+|\xi|^2+1=0.
	\end{aligned}
\end{equation*}
A simple computation gives
\begin{align*}
		\lambda_1=\frac{-1+\sqrt{1-4(|\xi|^{2}+1)}}{2}=\frac{-1+\sqrt{3+4|\xi|^{2}}i}{2}, \quad 
		\lambda_2=\frac{-1-\sqrt{1-4(|\xi|^{2}+1)}}{2}=\frac{-1-\sqrt{3+4|\xi|^{2}}i}{2}.
	\end{align*}
Then the explicit formula of $\hat{\Phi}=(\hat{n},\hat{q})^t$ can be expressed as
\begin{equation}\label{032501}
	\hat{\Phi}=  \left(\begin{array}{l}
		\hat{n} \\
		\hat{q}
	\end{array}\right)=\left(\begin{array}{ll}
		\frac{\lambda_1 e^{\lambda_2 t}-\lambda_2 e^{\lambda_1 t}}{\lambda_1-\lambda_2} & \frac{e^{\lambda_2 t}-e^{\lambda_1 t}}{\lambda_1-\lambda_2}|\xi| \\
		\frac{e^{\lambda_1 t}-e^{\lambda_2 t}}{\lambda_1-\lambda_2} \frac{1+|\xi|^2}{|\xi|} & \frac{\lambda_1 e^{\lambda_1 t}-\lambda_2 e^{\lambda_2 t}}{\lambda_1-\lambda_2}
	\end{array}\right)\left(\begin{array}{l}
		\hat{n}_0 \\
		\hat{q}_0
	\end{array}\right).
\end{equation}	
Similarly as the above, letting 			
$\hat{\Psi}=(\hat{w}, \hat{b})^t$ and $\hat{\Psi}_0=(\hat{w}_0, \hat{b}_0)^t$,
we obtain from \eqref{B} that
\begin{equation*}
	\left\{\begin{array}{l}
		\partial_t \hat{\Psi}+\hat{S_2} \hat{\Psi}=0, \\
		\hat{\Psi}|_{t=0}=\hat{\Psi}_0,
	\end{array}\right. \quad \mbox{with}\quad 	\hat{S_2}=\left(\begin{array}{cc}
		I & -I \\
		-I &(1+|\xi|^2) I
	\end{array}\right)
\end{equation*}
and the eigenvalues of $\hat{S_2}$ can be computed by solving
\begin{equation*}
	|\lambda I+\hat{S_2}|= \left(\lambda^2+(2+|\xi|^2) \lambda+|\xi|^2\right)^3=0.
\end{equation*}	
After directly calculation, we find
\begin{align*}
	\lambda_3=\frac{-\left(2+|\xi|^2\right)+\sqrt{4+|\xi|^ 4}}{2}  \quad \mbox{and} \quad \lambda_4=\frac{-\left(2+|\xi|^2\right)-\sqrt{4+|\xi|^ 4}}{2} 
\end{align*}
indicating that the eigenvalues $\lambda_3$ and $\lambda_4$ each have an algebraic multiplicity of $3$.
Then, $\hat{\Psi}$ can be explicitly expressed as
\begin{equation}\label{032502}
	\hat{\Psi}(t,\xi)=\left(\begin{array}{cc}\hat{w}  \\ \hat{b} \end{array} \right)=
	\left(\begin{array}{cc}
		\frac{(\lambda_3+1)e^{\lambda_4t}-(\lambda_4+1)e^{\lambda_3t}}{\lambda_3-\lambda_4}I& \frac{e^{\lambda_3t}-e^{\lambda_4t}}{\lambda_3-\lambda_4} I\\
		\frac{e^{\lambda_3t}-e^{\lambda_4t}}{\lambda_3-\lambda_4}I
		&\frac{(\lambda_3+1)e^{\lambda_3t}-(\lambda_4+1)e^{\lambda_4t}}{\lambda_3-\lambda_4}I
	\end{array}\right)\left(\begin{array}{cc}\hat{w}_0 \\ \hat{b}_0\end{array} \right).
\end{equation}

In the following lemma, we investigate the asymptotic behavior of the eigenvalues by dividing the frequency into the low-frequency and high-frequency parts.
\begin{lem}\label{THZ4}
	Assume $r_0$ is a sufficiently small positive constant. For the low frequency part $|\xi|<r_0$, the eigenvalues satisfy
	\begin{equation*}
		\begin{aligned}
		 &\lambda_3= -\frac{1}{2}|\xi|^2+O\left(|\xi|^4\right)  \quad \mbox{and} \quad \lambda_4= -2-\frac{1}{2}|\xi|^2+O\left(|\xi|^4\right).
		\end{aligned}
	\end{equation*}
	For $|\xi|\geq r_0$, there exists a positive constant $R$ such that
	\begin{equation*}
		{\rm{Re}}\,(\lambda_1,\lambda_2,\lambda_3,\lambda_4)\leq -R.
	\end{equation*}
\end{lem}
\begin{proof}
For $|\xi|< r_0 \ll 1$, we find
$$
\begin{aligned}
	\sqrt{4+|\xi|^ 4}=2 \sqrt{1+\frac{|\xi|^4}{4}}=2\left(1+\frac{1}{8}|\xi|^4 + O(|\xi|^8)\right),
\end{aligned}
$$
and thus
$$
\begin{aligned}
	& \lambda_3=\frac{-\left(2+|\xi|^2\right)+\sqrt{4+|\xi|^ 4}}{2}=-\frac{1}{2}|\xi|^2 + O\left(|\xi|^4\right), \\
	& \lambda_4=\frac{-\left(2+|\xi|^2\right)-\sqrt{4+|\xi|^ 4}}{2}=-2-\frac{1}{2}|\xi|^2 + O\left(|\xi|^4\right). 
	&
\end{aligned}
$$
On the other hand, for $|\xi|>R_0 \gg 1$, we get
$$
\begin{aligned}
\sqrt{4+|\xi|^ 4}=|\xi|^2 \sqrt{1+\frac{4}{|\xi|^4}}=|\xi|^2 \left(1+\frac{2}{|\xi|^4}+O(|\xi|^{-8})\right), 
\end{aligned}
$$			
and this implies
$$
\begin{aligned}
	& \lambda_3=\frac{-\left(2+|\xi|^2\right)+\sqrt{4+|\xi|^ 4}}{2}=-1+O\left(|\xi|^{-2}\right), \\
	& \lambda_4=\frac{-\left(2+|\xi|^2\right)-\sqrt{4+|\xi|^ 4}}{2}=-1-|\xi|^{2}+O\left(|\xi|^{-2}\right). \\
\end{aligned}
$$			
For $r_0\leq |\xi|\leq R_0$ , we easily obtain $ \text{Re}\,(\lambda_3,\lambda_4)\leq -\frac{r_{0}^{2}}{R_0^2+2}.$ 
Combining $	{\rm{Re}}~{\lambda_1}=\	{\rm{Re}}~{\lambda_2}=-\frac12$ and choosing $R=\min\left\{\frac{r_{0}^{2}}{R_0^2+2},\frac12\right\}$, we conclude for $|\xi|\geq r_0$ that
	$${\rm{Re}}\,(\lambda_1,\lambda_2,\lambda_3,\lambda_4)\leq -R.$$
This completes the proof.
\end{proof}

%
%
%
%
%
\subsubsection{Green function of the linearized system}\label{sssec_green}
To establish the time-decay estimates of the linear system, let $\eta=(n,u,v)^t$. Then, the system \eqref{Linear} can be expressed as 
\begin{equation}\label{eqn_lin}
	\left\{\begin{array}{l}
		\partial_t \eta(t)+\mathcal{L} \eta(t)=0, \\
		\eta|_{t=0}=\eta_0,
	\end{array} \right.
\end{equation} 
where the linear operator $\mathcal{L}$ is given by
\begin{equation}\label{S}
	\mathcal{L}=\left(\begin{array}{ccc}
		0 & \rm{div} & 0  \\
		\nabla+\nabla(-\Delta)^{-1} &  I  & -I \\
		0 &-\mathbb{P}& I -\Delta \\
	\end{array}\right).
\end{equation}
Taking Fourier transformation of \eqref{Linear} yields
\begin{equation*}
	\left\{\begin{array}{l}
		\partial_t \hat{\eta}(t)+\hat{\mathcal{L}} \hat{\eta}(t)=0, \\
		\hat{\eta}|_{t=0}=\hat{\eta}_0.
	\end{array} \right. 
\end{equation*}
Then, by solving the above equations, we have
$$
\hat{\eta}(t,\xi)=e^{-t\hat{\mathcal{L}}}\hat{\eta}_0=\hat{S}(t,\xi)\hat{\eta}_0,
$$ 
where the Green function $S(t,x)=e^{-t\mathcal{L}}$. We can directly obtain the time decay rates of the Green function $S(t,x)$ as follows:
\begin{prop}\label{lemL1}
	For a given function $g(t,x)=(g_1,g_2,g_3)^t$, we have the following decay estimates of the Green function $S(t,x)=(S_{ik})_{3\times3}$:
\begin{equation}\label{S(t)}
	\begin{aligned}
&\|\nabla^k S_{1i}*g_i\|_{L^2} \leq Ce^{-Rt}\|\nabla^k g_i\|_{L^2}, \quad i=1,2, \cr
&\|\nabla^k S_{21}*g_1\|_{L^2}\leq Ce^{-Rt}(\|\nabla^{k} (-\Delta)^{-1}\nabla g_1\|_{L^2}+\|\nabla^kg_1\|_{L^2}), \quad \mbox{and}\\
&\|\nabla^k S_{ij} *g_j\|_{L^2}\leq C(1+t)^{-\frac{3}{4}-\frac{k}{2}}\|g_j\|_{L^1}+Ce^{-Rt}\|\nabla^k g_j\|_{L^2}, \quad i,j=2,3.
	\end{aligned}
\end{equation}	
In particular, for the low-frequency part $S^\ell(t,x)=\mathcal{K}_1S(t,x)=(S^\ell_{ik})_{3\times 3}$, we have
\begin{equation}\label{S(t)2}
	\begin{aligned}
&\|\nabla^k S_{1i}^\ell*g_i\|_{L^2}\leq Ce^{-Rt}\|\nabla^k g_i\|_{L^2}, \quad i=1,2,  \cr
&\|\nabla^k S_{21}^\ell*g_1\|_{L^2}\leq Ce^{-Rt}\|\nabla^{k} (-\Delta)^{-1}\nabla g_1\|_{L^2}, \quad \mbox{and}\\
&\|\nabla^k S_{ij}^\ell *g_j\|_{L^2}\leq C(1+t)^{-\frac{3}{4}-\frac{k}{2}}\|g_j\|_{L^1},\quad i,j=2,3.
	\end{aligned}
\end{equation}		
\end{prop}
\begin{proof}
Note that the function $u$ can be decomposed into its divergence part and curl part as follows:
\begin{align*}
u&=-\Lambda^{-1}\nabla\Lambda^{-1}\text{div}u+\Lambda^{-1}\text{curl} ~(\Lambda^{-1}\text{curl}~u) =-\Lambda^{-1}\nabla q+\Lambda^{-1}\text{curl}~w.
\end{align*}
Then by taking Fourier transform and utilizing the expression of $(\hat{q},\hat{w})$ in \eqref{032501} and \eqref{032502}, we get 
\begin{equation*}
	\begin{aligned}
		 \hat{u}(t,\xi)&=-\frac{1}{|\xi|} i \xi \cdot\left(\frac{e^{\lambda_1 t}-e^{\lambda_2 t}}{\lambda_1-\lambda_2} \frac{1+|\xi|^2}{|\xi|} \hat{n}_0+\frac{\lambda_1 e^{\lambda_1 t}-\lambda_2 e^{\lambda_2 t}}{\lambda_1-\lambda_2} \hat{q}_0\right) \\
		&\quad +\frac{1}{|\xi|}\left(\begin{array}{ccc}
			0 & -i \xi_3 & i \xi_2 \\
			i \xi_3 & 0 & -i \xi_1 \\
			-i \xi_2 & i \xi_1 & 0
		\end{array}\right)\left[\frac{\left(\lambda_3+1\right) e^{\lambda_4 t}-\left(\lambda_4+1\right) e^{\lambda_3 t}}{\lambda_3-\lambda_4} \hat{w}_0+\frac{e^{\lambda_3 t}-e^{\lambda_4 t}}{\lambda_3-\lambda_4} \hat{b}_0\right] \\
		& =-\frac{1+|\xi|^2}{|\xi|^2} \frac{e^{\lambda_1 t}-e^{\lambda_2 t}}{\lambda_1-\lambda_2} i \xi\hat{n}_0+\frac{e^{\lambda_3 t}-e^{\lambda_4 t}}{\lambda_3-\lambda_4} \hat{v}_0\\
		&\quad+\left[\frac{\left(\lambda_3+1\right) e^{\lambda_4 t}-(\lambda_4+1) e^{\lambda_3 t}}{\lambda_3-\lambda_4}\left(I-\frac{\xi \xi^ t}{|\xi|^2}\right)+\frac{\lambda_1 e^{\lambda_1 t}-\lambda_2 e^{\lambda_2 t}}{\lambda_1-\lambda_2} \frac{\xi \xi^t}{|\xi|^2}\right] \hat{u}_0 .
	\end{aligned}
\end{equation*}
The Fourier transform of velocity $v$ can be treated similarly, giving:
\begin{equation*}		
\hat{v}(t,\xi)=\frac{e^{\lambda_3t}-e^{\lambda_4t}}{\lambda_3-\lambda_4}\Big(I-\frac{\xi\xi^t}{|\xi|^2}\Big)\hat{u}_0(\xi)+\frac{(\lambda_3+1)e^{\lambda_3t}-(\lambda_4+1)e^{\lambda_4t}}{\lambda_3-\lambda_4}\hat{v}_0(\xi).
\end{equation*}
In summary, we obtain
\begin{equation}\label{Main3}
	\left\{
	\begin{aligned}
		\hat{n}(t,\xi)&=\frac{\lambda_1 e^{\lambda_2 t}-\lambda_2 e^{\lambda_1 t}}{\lambda_1-\lambda_2} \hat{n}_0-\frac{e^{\lambda_1 t}-e^{\lambda_2 t}}{\lambda_1-\lambda_2}i\xi\cdot\hat{u}_0, \\
		\hat{u}(t,\xi)& =-\frac{1+|\xi|^2}{|\xi|^2} \frac{e^{\lambda_1 t}-e^{\lambda_2 t}}{\lambda_1-\lambda_2} i \xi \hat{n}_0 +\frac{e^{\lambda_3 t}-e^{\lambda_4 t}}{\lambda_3-\lambda_4} \hat{v}_0  \\
		&\quad +\left[\frac{\left(\lambda_3+1\right) e^{\lambda_4 t}-(\lambda_4+1) e^{\lambda_3 t}}{\lambda_3-\lambda_4}\left(I-\frac{\xi \xi^ t}{|\xi|^2}\right)+\frac{\lambda_1 e^{\lambda_1 t}-\lambda_2 e^{\lambda_2 t}}{\lambda_1-\lambda_2} \frac{\xi \xi^t}{|\xi|^2}\right] \hat{u}_0,\\
		\hat{v}(t,\xi)&=\frac{e^{\lambda_3t}-e^{\lambda_4t}}{\lambda_3-\lambda_4}\big(I-\frac{\xi\xi^t}{|\xi|^2}\big)\hat{u}_0(\xi)+\frac{(\lambda_3+1)e^{\lambda_3t}-(\lambda_4+1)e^{\lambda_4t}}{\lambda_3-\lambda_4}\hat{v}_0(\xi).
	\end{aligned}
	\right.
\end{equation}
The Fourier transform of Green function $S(t,x)$ is calculated as 
$$
\begin{aligned}
	&\hat{S}(t,\xi)\\
	&=
	\left(\begin{array}{ccc}
		\frac{\lambda_1e^{\lambda_2t}-\lambda_2e^{\lambda_1t}}{\lambda_1-\lambda_2}& -\frac{e^{\lambda_1t}-e^{\lambda_2t}}{\lambda_1-\lambda_2}i\xi^t &0\\
		-\frac{1+|\xi|^2}{|\xi|^2}\frac{e^{\lambda_1t}-e^{\lambda_2t}}{\lambda_1-\lambda_2} i\xi
		&\frac{(\lambda_3+1)e^{\lambda_4t}-(\lambda_4+1)e^{\lambda_3t}}{\lambda_3-\lambda_4}\big(I-\frac{\xi\xi^t}{|\xi|^2}\big)+\frac{\lambda_1e^{\lambda_1t}-\lambda_2e^{\lambda_2t}}{\lambda_1-\lambda_2}\frac{\xi\xi^t}{|\xi|^2}&\frac{e^{\lambda_3t}-e^{\lambda_4t}}{\lambda_3-\lambda_4}I\\
		0&\frac{e^{\lambda_3t}-e^{\lambda_4t}}{\lambda_3-\lambda_4t}\big(I-\frac{\xi\xi^t}{|\xi|^2}\big)&\frac{(\lambda_3+1)e^{\lambda_3t}-(\lambda_4+1)e^{\lambda_4t}}{\lambda_3-\lambda_4}I
	\end{array}\right)\\
	&\triangleq
	\left(\begin{array}{ccc}
		\widehat{S_{11}}&\widehat{S_{12}} &0\\
		\widehat{S_{21}}
		&\widehat{S_{22}}&\widehat{S_{23}}\\
		0&\widehat{S_{32}}&\widehat{S_{33}}
	\end{array}\right).
\end{aligned}
$$
Therefore, applying Lemma \ref{THZ4} yields the desired result \eqref{S(t)}. In particular, for the low-frequency part, we repeat the process in \eqref{S(t)} to obtain \eqref{S(t)2}. Indeed, since $|\xi|<r_0$, we find
	$$
	\lambda_1=-\frac{1}{2}+\frac{\sqrt{3}}{2}i+O(|\xi|^2)\quad \mbox{and} \quad \lambda_2=-\frac{1}{2}-\frac{\sqrt{3}}{2}i+O(|\xi|^2).
	$$
	For $|\xi|>R_0$, we get
	 $$
	 \lambda_1=-\frac{1}{2}+i|\xi|+O(|\xi|^{-1}) \quad \mbox{and} \quad  \lambda_2=-\frac{1}{2}-i|\xi|+O(|\xi|^{-1}).
	 $$
	 On the other hand, when $r_0\leq |\xi|\leq R_0$, we obtain
	 $$
	  \lambda_1=\frac{-1+\sqrt{3+4|\xi|^{2}}i}{2} \quad \mbox{and} \quad \lambda_2=\frac{-1-\sqrt{3+4|\xi|^{2}}i}{2},\quad \text{Re}~ \lambda_1=\text{Re}~\lambda_2=-\frac{1}{2}.
	 $$
	 For $|\xi|<r_0$, it is straightforward to estimate that 
	 \begin{align*}
	 \left| \frac{1}{\lambda_1-\lambda_2}\right|\leq C,~~
	\left| \frac{\lambda_1}{\lambda_1-\lambda_2}\right|\leq C,~~ \left|\frac{\lambda_2}{\lambda_1-\lambda_2}\right|\leq C,
	\end{align*}
	 and for $|\xi|\geq r_0$ that
	 \begin{align*}
	\left| \frac{1}{\lambda_1-\lambda_2}\right|\leq C,~~ \left| \frac{\lambda_1}{\lambda_1-\lambda_2}\right|\leq C,~~ \left|\frac{\lambda_2}{\lambda_1-\lambda_2}\right|\leq C.
	\end{align*}
Thus, we have
	\begin{align*}
		\|\nabla^kS_{11}*g_1\|_{L^2}&=\|(i\xi)^k\hat{S}_{11}\hat{g}_1\|_{L^2}\cr
		&\leq \left\|	\frac{\lambda_1e^{\lambda_2t}-\lambda_2e^{\lambda_1t}}{\lambda_1-\lambda_2}(i\xi)^k\hat{g}_1(\xi)\right\|_{L^2(|\xi|<r_0)}+\left\|	\frac{\lambda_1e^{\lambda_2t}-\lambda_2e^{\lambda_1t}}{\lambda_1-\lambda_2}(i\xi)^k\hat{g}_1(\xi)\right\|_{L^2(|\xi|\geq r_0)}\nonumber\\
		&\leq Ce^{-Rt}(\|(i\xi)^k\hat{g}_1(\xi)\|_{L^2(|\xi|<r_0)}+\|(i\xi)^k\hat{g}_1(\xi)\|_{L^2(|\xi|\geq r_0)})\nonumber\\
		&\leq Ce^{-Rt}\|\nabla^kg_1\|_{L^2}.
	\end{align*}	
The estimates for the other terms follow similarly, completing the proof.
\end{proof}

\subsubsection{Optimal temporal decay estimates for the linearized system}
In this part, we present both lower and upper bound estimates of decay rates for the velocities. Proposition \ref{lemL1} indicates that achieving a lower bound on the decay rate of convergence for the density is challenging since it decays exponentially fast towards the equilibrium state.
	
More precisely, the main result of this part can be stated as follows:
\begin{prop}\label{pro-li-d}
	Assume that all the conditions in Theorem \ref{thm3} hold. Let $(n,u,v)$ be the global classical solution to the linear system \eqref{eqn_lin}. Then for sufficiently large $t\geq t_0$, we have
	\begin{equation*}
		d_*(1+t)^{-\frac{3}{4}}\leq \|(u,v)(t)\|_{L^2}\leq C(1+t)^{-\frac{3}{4}}.
	\end{equation*}
Moreover, we obtain
	\begin{equation*}
		\bar{d}_*(1+t)^{-\frac{7}{4}}\leq \|(u-v)(t)\|_{L^2}\leq C(1+t)^{-\frac{7}{4}},
	\end{equation*}
	where $C$, $d_*$ and $\bar{d}_*$ are positive constants independent of time.
\end{prop}
\begin{proof}
Following Proposition \ref{lemL1}, we establish the upper bound estimates for the velocities $(u,v)$ as:
	\begin{equation*}
		\|(u,v)(t)\|_{L^2}\leq C(1+t)^{-\frac{3}{4}} \quad \mbox{and} \quad  \|(u-v)(t)\|_{L^2}\leq C(1+t)^{-\frac{7}{4}}.
	\end{equation*}
For the estimates of the lower bound of the time decay rates, by using the expression \eqref{Main3}, we divide $\hat{u}$ into two terms:
	\begin{equation*}
		\begin{aligned}
			\hat{u}(t,\xi) = I_1+I_2,
		\end{aligned}
	\end{equation*}
where
\[
I_1 \triangleq \frac{e^{\lambda_3 t}}{\lambda_3-\lambda_4} \hat{v}_0-\frac{(\lambda_4+1)e^{\lambda_3t}}{\lambda_3-\lambda_4}\big(I-\frac{\xi\xi^t}{|\xi|^2}\big)\hat{u}_0
\]
and
\[
I_2  \triangleq -\frac{1+|\xi|^2}{|\xi|^2} \frac{e^{\lambda_1 t}-e^{\lambda_2 t}}{\lambda_1-\lambda_2} i \xi \hat{n}_0 +\Big[\frac{(\lambda_3+1)e^{\lambda_4t}}{\lambda_3-\lambda_4}\big(I-\frac{\xi\xi^t}{|\xi|^2}\big)+\frac{\lambda_1e^{\lambda_1t}-\lambda_2e^{\lambda_2t}}{\lambda_1-\lambda_2}\frac{\xi\xi^t}{|\xi|^2}\Big]\hat{u}_0-\frac{e^{\lambda_4 t}}{\lambda_3-\lambda_4} \hat{v}_0.
\]	
We then use the assumption \eqref{in-data-optimal} to deduce
	\begin{align*}
		\|I_1\|_{L^2}^2
		&=\int_{\mathbb{R}^3}\Big|\frac{e^{\lambda_3t}}{\lambda_3-\lambda_4}\Big|^2\Big|\hat{v}_0-(\lambda_4+1)\Big(I-\frac{\xi\xi^t}{|\xi|^2}\Big)\hat{u}_0\Big|^2d\xi\\
		&\geq\int_{|\xi|<r_0}\Big|\frac{e^{\lambda_3t}}{\lambda_3-\lambda_4}\Big|^2\Big|\hat{v}_0-(\lambda_4+1)\Big(I-\frac{\xi\xi^t}{|\xi|^2}\Big)\hat{u}_0\Big|^2d\xi\\
		&\geq \frac{1}{4}\int_{|\xi|<r_0}e^{-|\xi|^2t}\Big|\Big(I-\frac{\xi\xi^t}{|\xi|^2}\Big)\hat{u}_0+\hat{v}_0\Big|^2d\xi\\
		&\geq \frac{1}{4}\alpha_0^2\int_{|\xi|<r_0}e^{-|\xi|^2t}d\xi\\
		&\geq C_9^2\alpha_0^2(1+t)^{-\frac{3}{2}},
	\end{align*}
	where $C_9$ is a small positive constant independent of time. As for $I_2$, it can be bounded by 
	$$
	\|I_2\|_{L^2}\leq  Ce^{-Rt}.
	$$
	Thus, for large $t$, we obtain
	\begin{equation}\label{LT3}
		\|u(t)\|_{L^2}=\|\hat{u}(t)\|_{L^2}\geq \|I_1\|_{L^2}-\|I_2\|_{L^2}
		\geq C_9\alpha_0(1+t)^{-\frac{3}{4}}-Ce^{-Ct}\geq \frac{1}{2}C_9\alpha_0(1+t)^{-\frac{3}{4}}.
	\end{equation}
	
Similarly, we decompose $\hat{v}(t,\xi)$ as:	
	\begin{equation*}
		\begin{aligned}
			\hat{v}(t,\xi)
			&=\Big\{\frac{e^{\lambda_3t}}{\lambda_3-\lambda_4} \Big(I-\frac{\xi\xi^t}{|\xi|^2}\Big)\hat{u}_0(\xi)+\frac{e^{\lambda_3t}}{\lambda_3-\lambda_4}\hat{v}_0(\xi)\Big\}\\
			&\quad +\Big\{-\frac{e^{\lambda_4t}}{\lambda_3-\lambda_4}\Big(I-\frac{\xi\xi^t}{|\xi|^2}\Big)\hat{u}_0(\xi)+\big[ \frac{\lambda_3e^{\lambda_3t}-(\lambda_4+1)e^{\lambda_4t}}{\lambda_3-\lambda_4}\big]\hat{v}_0(\xi)\Big\}\\
			&\triangleq I_3+I_4,
			\end{aligned}
	\end{equation*}
where
	$$
	\begin{aligned}
	\|I_3\|_{L^2}^2&=\int_{\mathbb{R}^3}\Big| \frac{e^{\lambda_3t}}{\lambda_3-\lambda_4}\Big|^2 \Big| \Big(I-\frac{\xi\xi^t}{|\xi|^2}\Big)\hat{u}_0+\hat{v}_0\Big|^2d\xi\\
	&\geq \int_{|\xi|<r_0}\Big| \frac{e^{\lambda_3t}}{\lambda_3-\lambda_4}\Big|^2\Big| \Big(I-\frac{\xi\xi^t}{|\xi|^2}\Big)\hat{u}_0+\hat{v}_0\Big|^2d\xi\\
	&\geq \frac{1}{4}\alpha_0^2\int_{|\xi|<r_0}e^{-|\xi|^2t}d\xi\\
	&\geq C_9^2\alpha_0^2(1+t)^{-\frac{3}{2}},
	\end{aligned}
	$$
	due to \eqref{in-data-optimal}. On the other hand, it follows from Lemma \ref{THZ4} and direct calculations that 
	$$
	\|I_4\|_{L^2}\leq C(1+t)^{-\frac{7}{4}}+Ce^{-Rt}\leq C(1+t)^{-\frac{7}{4}}.
	$$
Consequently, for large time $t > 0$, we deduce 
	\begin{equation}\label{LT2}
			\|v(t)\|_{L^2}\geq \|I_3\|_{L^2}-\|I_4\|_{L^2}\geq \frac{1}{2}C_9\alpha_0(1+t)^{-\frac{3}{4}}.
	\end{equation}
 Combining \eqref{LT3} and \eqref{LT2} yields  
	\begin{equation*}
	d_*(1+t)^{-\frac{3}{4}}\leq \|(u,v)(t)\|_{L^2}\leq C(1+t)^{-\frac{3}{4}},
	\end{equation*}
	where $C$  and $d_*=\frac{1}{2}C_9\alpha_0$ are positive constants independent of time. 
	
For the decay estimate of the difference of velocities $u-v$, we consider
\[
	\hat{u}-\hat{v} =  I_5+I_6,
\]
where
\[
I_5 \triangleq -\frac{1+|\xi|^2}{|\xi|^2} \frac{e^{\lambda_1 t}-e^{\lambda_2 t}}{\lambda_1-\lambda_2} i \xi  \hat{n}_0 +\frac{\lambda_1 e^{\lambda_1 t}-\lambda_2 e^{\lambda_2 t}}{\lambda_1-\lambda_2} \frac{\xi \xi^t}{|\xi|^2}\hat{u}_0+\frac{(\lambda_3+2)e^{\lambda_4t}}{\lambda_3-\lambda_4}\left(I-\frac{\xi \xi^ t}{|\xi|^2}\right)\hat{u}_0+\frac{\lambda_4e^{\lambda_4t}}{\lambda_3-\lambda_4}\hat{v}_0
\]
and
\[
I_6 \triangleq -\frac{(\lambda_4+2) e^{\lambda_3 t}}{\lambda_3-\lambda_4}\left(I-\frac{\xi \xi^ t}{|\xi|^2}\right)\hat{u}_0-\frac{\lambda_3e^{\lambda_3t}}{\lambda_3-\lambda_4}\hat{v}_0.
\]
Then, by Lemma \ref{THZ4}, we have 
\[
\|I_6\|\leq C\left\| |\xi|^2e^{-|\xi|^2t} \right\|_{L^2}+Ce^{-Rt}\leq C(1+t)^{-\frac{7}{4}}\quad \mbox{and} \quad \|I_5\|\leq Ce^{-Rt},
\]
which implies 
\[
\|(u-v)(t)\|_{L^2}=\|(\hat{u}-\hat{v})(t)\|_{L^2}\leq \|I_5\|_{L^2}+\|I_6\|_{L^2}\leq C(1+t)^{-\frac{7}{4}}.
\]
On the other hand, by Proposition \ref{THZ4} and the assumption \eqref{in-data-optimal}, we find
\begin{align*}
	\|I_6\|_{L^2}^2&\geq \int_{|\xi|<r_0}
	\left|-\frac{(\lambda_4+2) e^{\lambda_3 t}}{\lambda_3-\lambda_4}\Big(I-\frac{\xi \xi^ t}{|\xi|^2}\Big)\hat{u}_0-\frac{\lambda_3e^{\lambda_3t}}{\lambda_3-\lambda_4}\hat{v}_0(\xi)\right|^2d\xi\\
	&\geq \int_{|\xi|< r_0}
	\Big|\frac{1}{4}|\xi|^2e^{\lambda_3t}\Big|^2\left|\Big(I-\frac{\xi \xi^ t}{|\xi|^2}\Big)\hat{u}_0+\hat{v}_0\right|^2d\xi\\
   &\geq \frac{1}{16}\alpha_0^2 \int_{|\xi|<r_0}
    |\xi|^4e^{-|\xi|^2t}d\xi\\
	&\geq C_{10}^2\alpha_0^2(1+t)^{-\frac{7}{2}},
\end{align*}
where $C_{10}$ denotes a small positive constant independent of time. Thus, we obtain
$$
\|u-v\|_{L^2}=\|\hat{u}-\hat{v}\|_{L^2}\geq \|I_6\|_{L^2}-\|I_5\|_{L^2}
\geq C_{10}\alpha_0(1+t)^{-\frac{7}{4}}-Ce^{-Ct}\geq \frac{1}{2}C_{10}\alpha_0(1+t)^{-\frac{7}{4}}.
$$
Choosing $\bar{d}_*=\frac{1}{2}C_{10}\alpha_0$, we conclude
\begin{align*}
\bar{d}_*(1+t)^{-\frac{7}{4}}\leq 	\|(u-v)(t)\|_{L^2}\leq C(1+t)^{-\frac{7}{4}}.
	\end{align*}
	Thus we complete the proof.
\end{proof}

%
%
%
%
%
%
	\subsection{Proof of Theorem \ref{thm1}: Large-time behavior estimates}\label{Sec4}
In this subsection, we provide the upper bound of the decay rates of solutions to the EP-NS system \eqref{Main2} proving the large-time behavior estimates stated in Theorem \ref{thm1}. We consider the following nonlinear system:
\begin{equation*}
	\left\{\begin{array}{l}
		\partial_t \eta(t)+\mathcal{L}\eta(t)=F, \\
		\eta|_{t=0}=\eta_0,
	\end{array}\right.
\end{equation*}
with $\eta=(n,u,v)^t$, $\mathcal{L}$ appeared in \eqref{S} ,and $F=(f_1,f_2,f_3)^t$, where $f_i$, $i=1,2,3$ are given as in \eqref{nonlinearterms}. By Duhamel's principle, we have
\begin{equation*}
		\eta(t)=  S(t) * \eta_0+\int_0^t S(t-s) * F(s) \,ds, 
\end{equation*}
where $S=e^{-t\mathcal{L}}$. From this, we obtain
\begin{equation}\label{eqn_nuv}
	\begin{aligned}
		n(t)&=  S_{11} * n_0+S_{12} * u_0 \\
		&\quad +\int_0^t S_{11}(t-s) * f_1(s) \,ds+\int_0^t S_{12}(t-s) * f_2(s) \,ds ,\\
		u(t)&=  S_{21} * n_0+S_{22} * u_0+S_{23} * v_0 \\
		& \quad +\int_0^t S_{21}(t-s) * f_1(s) \,ds+\int_0^t S_{22}(t-s)* f_2(s) \,ds+\int_0^t S_{23}(t-s) * f_3(s) \,ds,\\
		v(t)&=  S_{32} * u_0+S_{33} * v_0+\int_0^t S_{32}(t-s)* f_2(s) \,ds+\int_0^t S_{33}(t-s) * f_3(s) \,ds.
	\end{aligned}
\end{equation}
We then recall the time-weighted energy function defined as
\[
\begin{aligned}
M(t)&= \sup _{0 \leq s \leq t} \Big\{\sum_{k=0}^2(1+s)^{\frac{11}{4}+\frac{k}{2}}\left\|\nabla^{k}n(s)\right\|_{L^2}
+(1+s)^{\frac{15}{4}}\|\nabla^3n(s)\|_{L^2}\\
&\hspace{7cm}+\sum_{k=0}^3(1+s)^{\frac{3}{4}+\frac{k}{2}}\left\|\nabla^{k}(u, v)(s)\right\|_{L^2}\Big\}.
\end{aligned}
\]
It's straightforward to verify that
\begin{equation}\label{MT}
\begin{aligned}
 \left\|\nabla^{k}n(t)\right\|_{L^2}&\leq(1+t)^{-\frac{11}{4}-\frac{k}{2}} M(t), \quad k=0,1,2, \\
	 \left\|\nabla^{3}n(t)\right\|_{L^2}&\leq(1+t)^{-\frac{15}{4}} M(t),\\
 \left\|\nabla^{k}(u, v)(t)\right\|_{L^2}&\leq(1+t)^{-\frac{3}{4}-\frac{k}{2}} M(t),\quad k=0,1,2,3. 
\end{aligned}
\end{equation}
Our next objective is to show that $M(t)$ remains uniformly bounded over time. More precisely, we prove the following proposition, which corresponds to the large-time behavior estimates of solutions stated in Theorem \ref{thm1}.
\begin{prop}\label{lemA1}
	 Let $(n,u,v)$ be the classical solutions of system \eqref{main4}. Then it holds that
	\begin{equation}\label{lemA1-1}
		\begin{aligned}
\left\|n(t) \right\|_{\dot{H}^{-1}}&\leq C\delta_0(1+t)^{-2},\\ 			
 \left\|\nabla^{k}n(t)\right\|_{L^2}&\leq C\delta_0(1+t)^{-\frac{11}{4}-\frac{k}{2}}, &&k=0,1,2, \\
\left\|\nabla^{3}n(t)\right\|_{L^2}&\leq C \delta_0(1+t)^{-\frac{15}{4}},\\
\left\|\nabla^{k}(u, v)(t)\right\|_{L^2}&\leq C\delta_0(1+t)^{-\frac{3}{4}-\frac{k}{2}},&& k=0,1,2,3,\\
	\|\nabla^k(u-v)(t)\|_{L^2}&\leq C\delta_0(1+t)^{-\frac{7}{4}-\frac{k}{2}}, &&k=0,1,
\end{aligned}
	\end{equation}
where $C$ is a positive constant independent of time.
\end{prop}
\begin{proof}
We first show the decay estimates of density. By Proposition \ref{lemL1}, we deduce
\begin{equation}\label{011703}
	\begin{aligned}	
		&\left\|\nabla^{k}n(t)\right\|_{L^2} \cr
		&\quad \leq\left\|\nabla^{k} S_{11} * n_0\right\|_{L^2}+\left\|\nabla^{k}S_{12} * u_0\right\|_{L^2}\\
		 &\qquad +\int_0^t\left\|\nabla^{k} S_{11}(t-s) * f_1(s)\right\|_{L^2} ds+\int_0^t\left\|\nabla^k S_{12}(t-s) * f_2(s)\right\|_{L^2}  ds \\
		&\quad \leq Ce^{-R t}\left(\left\|\nabla^{k} n_0\right\|_{L^2}+\left\|\nabla^{k} u_0\right\|_{L^2}\right)+C\int_0^t e^{-R(t-s)}\left\|\nabla^{k}\left(f_1, f_2\right)(s)\right\|_{L^2} ds \\
		&\quad \leq Ce^{-Rt}\left(\left\|\nabla^{k} n_0\right\|_{L^2}+\left\|\nabla^{k} u_0\right\|_{L^2}\right)+C\int_0^t e^{-R(t-s)}\Big\{\left\|\nabla^{k}(u \cdot \nabla n)\right\|_{L^2}+\left\|\nabla^{k}(u \cdot \nabla u)\right\|_{L^2}\\
	 	&\hspace{7.5cm} +\left\|\nabla^{k}(\nabla (-\Delta)^{-1}(e^n-1-n))\right\|_{L^2}\Big\}\,ds.
	\end{aligned}
\end{equation}
$\bullet$  For $k=0$, we obtain
\begin{equation}\label{011701}
	\begin{aligned}
		&\int_0^t e^{-R(t-s)}(\left\|u \cdot \nabla n\right\|_{L^2}+\|u \cdot \nabla u\|_{L^2}) \,ds \\
		&\quad\leq  \int_0^t e^{-R(t-s)}(\|u\|_{L^\infty}\|\nabla n\|_{L^2}+\|u\|_{L^\infty}\|\nabla u\|_{L^2}) \,ds \\
		&\quad\leq CM(t)^2\int_0^t e^{-R(t-s)}\Big\{(1+s)^{-\frac{3}{2}}(1+s)^{-\frac{13}{4}} +(1+s)^{-\frac{3}{2}}(1+s)^{-\frac{5}{4}} \Big\}\,ds \\
		&\quad\leq C(1+t)^{-\frac{11}{4}} M(t)^2.
	\end{aligned}
\end{equation}
Letting $Z(x,t)\triangleq e^n-1-n$, we also estimate
\begin{equation}\label{011702}
	\begin{aligned}
		\int_0^t e^{-R(t-s)}\left\|\nabla(-\Delta)^{-1}\left(e^n-1-n\right)\right\|_{L^2}  ds
			& =\int_0^t e^{-R(t-s)}\|\nabla K * Z\|_{L^2} \,ds \\
			& \leq\int_0^t e^{-R(t-s)}\left(\left\|\nabla K^\ell * Z\right\|_{L^2}+\left\|\nabla K^h * Z\right\|_{L^2}\right)  ds \\
			& \leq C\int_0^t e^{-R(t-s)}\left(\left\|\nabla K^\ell\right\|_{L^2}\|Z\|_{L^1}+\|Z\|_{L^2}\right)  ds \\
			& \leq C\int_0^t e^{-R(t-s)}\left(\|n\|_{L^2}^2+\|n\|_{L^{\infty}}\|n\|_{L^2}\right)  ds \\
			& \leq C M(t)^2 \int_0^t e^{-R(t-s)}\Big((1+s)^{-\frac{11}{2}}+(1+s)^{-\frac{25}{4}}\Big) \,ds\\
			 &\leq C(1+t)^{-\frac{11}{2}} M(t)^2,
		\end{aligned}
\end{equation}
where $K$ is the fundamental solution of the Laplace equation and we used $\|\nabla K^\ell\|_{L^2}\leq C$ and the fact that  $|\widehat{\nabla K^h}|\leq C$. Then we substitute \eqref{011701} and \eqref{011702} into \eqref{011703} to get the time decay rates for $k=0$ that
$$
	\|n(t)\|_{L^2} \leq Ce^{-R t} \delta_0+C(1+t)^{-\frac{11}{4}} M(t)^2.
$$
$\bullet$  For $k=1,2$, we notice that
$$ 
\begin{aligned}
\|\nabla (u\cdot \nabla n)\|_{L^2} &\leq C( \| u \|_{L^\infty} \|\nabla^{2}n\|_{L^2}+\|\nabla u \|_{L^\infty} \|\nabla n\|_{L^2})\leq C(1+t)^{-\frac{21}{4}}M(t)^2,\\
 \|\nabla (u\cdot \nabla u)\|_{L^2}&\leq C(\|u \|_{L^\infty} \|\nabla^{2}u\|_{L^2}+\|\nabla u \|_{L^\infty} \|\nabla u\|_{L^2})\leq C(1+t)^{-\frac{13}{4}}M(t)^2,\\
 \|\nabla^{2} (u\cdot \nabla n)\|_{L^2}&\leq C( \| u \|_{L^\infty} \|\nabla^{3}n\|_{L^2}+\|\nabla u \|_{L^\infty} \|\nabla^{2} n\|_{L^2}+\|\nabla^2u\|_{L^2}\|\nabla n\|_{L^\infty})\leq C(1+t)^{-\frac{23}{4}}M(t)^2,\\
 \|\nabla^{2} (u\cdot \nabla u)\|_{L^2} &\leq C(\|\nabla u \|_{L^\infty} \|\nabla^{2}u\|_{L^2}+\| u \|_{L^\infty} \|\nabla^{3}u\|_{L^2})\leq C(1+t)^{-\frac{15}{4}}M(t)^2.
 \end{aligned}
 $$
 By using these estimates, we have the following estimate of derivatives of density:
\begin{equation*}
	\|\nabla ^kn(t)\|_{L^2}  \leq Ce^{-R t} \delta_0+C(1+t)^{-\frac{11}{4}-\frac{k}{2}} M(t)^2,\quad k=0,1,2. 
\end{equation*}
It is worth noting that we cannot directly take the third derivatives on the nonlinear term $u\cdot\nabla u$ due to the constraints of our setting with $M(t)$. To this end, we decompose it into low-high frequency parts:
\begin{align*}
n= n^\ell+ n^h.
\end{align*}
For $\nabla^3 n^\ell$, we find
\begin{equation}\label{032004}
	\begin{aligned}	
		&\left\|\nabla^{3} n^\ell(t)\right\|_{L^2} \cr
		&\quad \leq\left\|\nabla^{3} S^\ell_{11} * n_0\right\|_{L^2}+\left\|\nabla^{3}S^\ell_{12} * u_0\right\|_{L^2}\\
		&\qquad +\int_0^t\left\|\nabla^{3} S^\ell_{11}(t-s) * f_1(s)\right\|_{L^2} \,ds+\int_0^t\left\|\nabla^{3} S^\ell_{12}(t-s) * f_2(s)\right\|_{L^2} ds \\
		&\quad\leq Ce^{-R t}\left(\left\| n_0\right\|_{L^2}+\left\| u_0\right\|_{L^2}\right)+C\int_0^t e^{-R(t-s)}\left\|\nabla^{2}\left(f_1, f_2\right)(s)\right\|_{L^2} ds \\
		&\quad\leq Ce^{-Rt}\left(\left\| n_0\right\|_{L^2}+\left\| u_0\right\|_{L^2}\right)+C\int_0^t e^{-R(t-s)}\Big\{\left\|\nabla^{2}(u \cdot \nabla n)\right\|_{L^2}+\left\|\nabla^{2}(u \cdot \nabla u)\right\|_{L^2}\\
		&\hspace{7.5cm} +\left\|\nabla^{2}(\nabla (-\Delta)^{-1}(e^n-1-n)) \right\|_{L^2}\Big\}\,ds.
	\end{aligned}
\end{equation}
Here we estimate
$$ 
	\|\nabla^{2} (u\cdot \nabla n)\|_{L^2}\leq C(1+t)^{-\frac{23}{4}}M(t)^2,\quad 
	\|\nabla^{2} (u\cdot \nabla u)\|_{L^2}\leq C(1+t)^{-\frac{15}{4}}M(t)^2,
$$
and 
\begin{equation*}
	\begin{aligned}
		\left\|\nabla^2 \nabla(-\Delta)^{-1}\left(e^n-1-n\right)\right\|_{L^2} & \leq\left\|\nabla^2 \nabla K *\left(e^n-1-n\right)\right\|_{L^2} \\
		& \leq C\left\|\nabla^2\left(e^n-1-n\right)\right\|_{L^1}+C\left\|\nabla^2\left(e^n-1-n\right)\right\|_{L^2} \\
		& \leq C\left(\| \nabla^2(n \cdot n) \|_{L^1}+\| \nabla^2(n \cdot n) \|_{L^2}\right) \\
		& \leq C(1+t)^{-\frac{13}{2}}M(t)^2.
	\end{aligned}
\end{equation*}
Substituting the above facts into \eqref{032004}, we have
\begin{equation}  \label{011707}
		\left\|\nabla^3 n^\ell(t)\right\|_2  \leq Ce^{-R t} \delta_0+C(1+t)^{-\frac{15}{4}} M(t)^2.
\end{equation}	

Now we move to the estimate of the decay rates for velocities. 
By using the fact
\[
\big(I-\frac{\xi\xi^t}{|\xi|^2}\big) \cdot \frac{i\xi}{|\xi|^2} =0
\]
and  Proposition \ref{lemL1}, we obtain from \eqref{eqn_nuv} that 
\begin{equation*}
	\begin{aligned}
		 \|u(t)\|_{L^2}
		 &\leq Ce^{-R t}(\|n_0\|_{\dot{H}^{-1}}+\|n_0\|_{L^2})
		 +C(1+t)^{-\frac{3}{4}}\|(u_0,v_0)\|_{L^1} +Ce^{-Rt}\|(u_0,v_0)\|_{L^2}\\
		&\quad + C \int_0^t e^{-R(t-s)} (\left\|\nabla(-\Delta)^{-1}f_1(s)\right\|_{L^2}+\|f_1(s)\|_{L^2}+\|f_2(s)\|_{L^2}+\|f_3(s)\|_{L^2})\,ds\\
		&\quad + C \int_0^t (1+ t-s)^{-\frac{3}{4}}(\|h_2(s)\|_{L^1}+\|h_3(s)\|_{L^1})\,ds,
	\end{aligned}
  \end{equation*}
where the nonlinear terms $f_1, f_2$ are given by \eqref{nonlinearterms} and the new definitions $h_2$ and $h_3$ satisfy 
\begin{equation}\label{h1}
		h_2=-u \cdot \nabla u,\quad  h_3=-v \cdot \nabla v+(e^n-1)(u-v).
\end{equation}
On the other hand, we use \eqref{MT} to get
\begin{equation}\label{042001}
	\begin{aligned}
&\|f_1(t)\|_{L^1}\leq C(1+t)^{-4}M(t)^2, \quad \|f_1(t)\|_{L^2}\leq  C(1+t)^{-\frac{19}{4}}M(t)^2,\\
&\|f_2(t)\|_{L^1}\leq C(1+t)^{-2}M(t)^2,\quad \|f_2(t)\|_{L^2}\leq C(1+t)^{-\frac{11}{4}}M(t)^2,\\
&\|h_2(t)\|_{L^1}\leq C(1+t)^{-2}M(t)^2,\quad \|h_3(t)\|_{L^1}\leq C(1+t)^{-2}M(t)^2,\\
&\|f_3(t)\|_{L^2}\leq C(1+t)^{-\frac{11}{4}}M(t)^2,
\end{aligned}
\end{equation}
and
$$
\left\|\nabla(-\Delta)^{-1}f_1(s)\right\|_{L^2}\leq C(\|f_1\|_{L^1}+\|f_1\|_{L^2})\leq C(1+s)^{-4}M(t)^2.
$$
Then, this together with applying H\"{o}lder inequality, we deduce
\begin{equation*}
	\begin{aligned}
	 \|u(t)\|_{L^2}&\leq C(1+t)^{-\frac34}\delta_0 +CM(t)^2\int_0^t e^{-R(t-s)} (1+s)^{-\frac{11}{4}}\,ds\\
		&\quad +CM(t)^2\int_0^t(1+t-s)^{-\frac{3}{4}}(1+s)^{-2}\,ds\\
&\leq C(1+t)^{-\frac{3}{4}}\left(\delta_0+M(t)^2\right).
	\end{aligned}
\end{equation*}
The derivatives of $u$ can be obtained in a similar manner. Thus, we get for $k=0,1,2$ that
$$
\left\|\nabla^{k}u(t)\right\|_{L^2} \leq C(1+t)^{-\frac{3}{4}-\frac{k}{2}}\left(\delta_0+M(t)^2\right).
$$
By using a similar argument, we also obtain the decay estimates of velocity $v$ for $k=0,1,2$ that
$$
\left\|\nabla^{k}v(t)\right\|_{L^2} \leq C(1+t)^{-\frac{3}{4}-\frac{k}{2}}\left(\delta_0+M(t)^2\right).
$$
For low-frequency part of $\left\|\nabla^3 u^\ell\right\|_{L^2}$, we consider
\begin{equation*}
	\begin{aligned}
		\nabla^{3}u^\ell(t)
		&= \nabla^{3}S^\ell_{21} * n_0+	\nabla^{3}S^\ell_{22} * u_0+	\nabla^{3}S^\ell_{23} * v_0 \\
		&\quad +\int_0^t 	\nabla^{3}S^\ell_{21}(t-s) * f_1(s) \,ds+\int_0^t 	\nabla^{3}S^\ell_{22}(t-s)* f_2(s) \,ds+\int_0^t 	\nabla^{3}S^\ell_{23}(t-s) * f_3(s) \,ds.
	\end{aligned}
\end{equation*}
Then we have 
\begin{equation}\label{032001}
	\begin{aligned}
		 \|\nabla^{3}u^\ell(t)\|_{L^2}
		 &\leq C(1+t)^{-\frac{9}{4}}\delta_0
	+C \int_0^t e^{-R(t-s)} \left\|\nabla^{2}f_1\right\|_{L^2} ds\cr
	&\quad +C \int_{0}^{\frac{t}{2}} (1+ t-s)^{-\frac{9}{4}}(\left\|h_2\right\|_{L^1}+\left\|h_3\right\|_{L^1}) \,ds \\
		&\quad +C \int_{\frac{t}{2}}^{t} (1+ t-s)^{-\frac{5}{4}}(\left\|\nabla^{2}h_2\right\|_{L^1}+\left\|\nabla^{2}h_3\right\|_{L^1}) \,ds\\
		&\leq C(1+t)^{-\frac{9}{4}}\left(\delta_0+M(t)^2\right),
	\end{aligned}
\end{equation}
where we used \eqref{042001} and the fact that 
\begin{align*}
\|\nabla^2f_1\|_{L^2}\leq C(1+t)^{-\frac{23}{4}}M(t)^2,\quad \left\|\nabla^{2}h_2\right\|_{L^1}\leq C(1+t)^{-3}M(t)^2,\quad \|\nabla^2h_3\|_{L^1}\leq C(1+t)^{-3}M(t)^2.
\end{align*}
 Similarly, we also achieve that 
\begin{equation}\label{032002}
\|\nabla^{3}v^\ell(t)\|_{L^2} \leq C(1+t)^{-\frac{9}{4}}\left(\delta_0+M(t)^2\right).
\end{equation}
For the decay rate of $\|\nabla^3(n^h,u^h,v^h)\|_{L^2}$, we consider the  high-frequency part of the main system \eqref{main4} as:
\begin{equation*}
	\left\{
	\begin{aligned}
		&\partial_t n^{h} +{\rm{div}}u^{h}=f^{h}_1, \\
		&\partial_t u^{h}+\nabla n^{h}+u^{h}-v^{h}+\nabla (-\Delta )^{-1}n^{h}=f^{h}_2,\\
		&\partial_t v^{h}+v^{h}-\mathbb{P}u^{h}-\Delta v^{h}=f^{h}_3,\\
		&\text{div} v^h=0.
	\end{aligned}
	\right.
\end{equation*}
Taking $L^2$ inner product with $\nabla^{3} (n^{h},u^{h},v^{h})$ in $\mathbb{R}^3$ gives
\begin{equation*}
	\begin{aligned}
		& \frac{1}{2} \frac{d}{d t}\|\nabla^{3} (n^{h},u^{h},v^{h})\|_{L^2}^2+\left\|\nabla^4 v^{h}\right\|_{L^2}^2 
		 +\|\nabla^{3}(u^{h}-v^{h})\|_{L^2}^2+\int_{\mathbb{R}^3} \nabla^{3} \nabla(-\Delta)^{-1} n^{h} \cdot  \nabla^{3} u^{h} \,dx\\
		 & \quad=\int_{\mathbb{R}^3} \nabla^{3} f^{h}_1 \cdot \nabla^{3} n^{h} \,dx+\int_{\mathbb{R}^3} \nabla^{3} f^{h}_2 \cdot \nabla^{3} u^{h} \,dx+\int_{\mathbb{R}^3} \nabla^{3}f^{h}_{3} \cdot \nabla^{3} v^{h} \,dx.
	\end{aligned}
\end{equation*}
Note that
\begin{equation*}
\begin{aligned}
	\int_{\mathbb{R}^3} \nabla^{3} \nabla(-\Delta)^{-1} n^{h} \cdot \nabla^{3} u^{h} \,dx 
	= & -\int_{\mathbb{R}^3} \nabla^{3}(-\Delta)^{-1} n^{h} \cdot \nabla^{3} \operatorname{div} u^{h} \,dx \\
	= & \frac12 \frac{d}{dt}\|\nabla^{4}(-\Delta)^{-1}n^{h}\|_{L^2}^{2}-\int_{\mathbb{R}^3} \nabla^{3} (-\Delta)^{-1}n^{h} \cdot \nabla^{3}f^{h}_1\,dx.
\end{aligned}
\end{equation*}
Then we obtain
\begin{equation}\label{E21}
	\begin{aligned}
		& \frac{1}{2} \frac{d}{d t}\Big\{\|\nabla^{3} (n^{h},u^{h},v^{h})\|_{L^2}^2+\|\nabla^{4}(-\Delta)^{-1}n^{h}\|_{L^2}^{2}\Big\}+\left\|\nabla^4 v^{h}\right\|_{L^2}^2 
		+\|\nabla^{3}(u^{h}-v^{h})\|_{L^2}^2\\
		& \quad=\int_{\mathbb{R}^3} \nabla^{3} f^{h}_1 \cdot \nabla^{3} n^{h} \,dx+\int_{\mathbb{R}^3} \nabla^{3} f^{h}_2 \cdot \nabla^{3} u^{h} \,dx\\
		&\qquad+\int_{\mathbb{R}^3} \nabla^{3}f^{h}_{3} \cdot \nabla^{3} v^{h} \,dx+\int_{\mathbb{R}^3} \nabla^{3} (-\Delta)^{-1}n^{h}  \cdot \nabla^{3}f^{h}_1\,dx.
	\end{aligned}
\end{equation}
We split the first term on the right-hand side of \eqref{E21} into two terms:
\begin{equation*}
	\int_{\mathbb{R}^3}\nabla^{3} f^{h}_1 \cdot \nabla^{3} n^{h} \,dx= \int_{\mathbb{R}^3} \nabla^{3} f_1 \cdot \nabla^{3} n^{h} \,dx-\int_{\mathbb{R}^3} \nabla^{3} f^\ell_1 \cdot \nabla^{3} n^{h} \,dx
	\triangleq J_{1}+J_{2}.
\end{equation*}
We then estimate $J_1$ and $J_2$ as
The term $J_1$ is estimated as 
	\begin{align*}
	J_1&\leq \Big|\int_{\mathbb{R}^3} (\nabla^{3} (u \cdot \nabla n)-u\cdot\nabla^3\nabla n) \cdot \nabla^{3} n^{h} \,dx\Big| +\Big|\int_{\mathbb{R}^3} u\cdot\nabla^3\nabla n \cdot \nabla^{3} n^{h} \,dx\Big| \\
	&\leq \Big|\int_{\mathbb{R}^3} (\nabla^{3} (u \cdot \nabla n)-u\cdot\nabla^3\nabla n) \cdot \nabla^{3} n^{h} \,dx\Big| +\Big|\int_{\mathbb{R}^3} u\cdot\nabla^3\nabla n^h \cdot \nabla^{3} n^{h} \,dx\Big|+\Big|\int_{\mathbb{R}^3} u\cdot\nabla^3\nabla n^\ell \cdot \nabla^{3} n^{h} \,dx\Big|  \\
	&\leq C(\|\nabla u\|_{L^\infty}\|\nabla^3n\|_{L^2}+\|\nabla^3u\|_{L^2}\|\nabla n\|_{L^\infty})\|\nabla^3n^h\|_{L^2}+C\|\text{div}u\|_{L^\infty}\|\nabla^3n^h\|_{L^2}\cr
	&\quad +C\|u\|_{L^\infty}\|\nabla^4n^\ell\|_{L^2}\|\nabla^3n^h\|_{L^2}\\
	 &\leq C\|\nabla u\|_{L^\infty}\|\nabla^3n\|_{L^2}^2+C\|\nabla n\|_{L^\infty}\|\nabla^3u\|_{L^2}\|\nabla^3n\|_{L^2}+C\|u\|_{L^\infty}\|\nabla^3n\|_{L^2}^2\\
	 &\leq C (1+t)^{-10} M(t)^3
	\end{align*}
and 
 \begin{equation*}
 	\begin{aligned}
 		J_{2}
 		\leq \|\nabla^{3}n^{h}\|_{L^2}\|\nabla^{3} f^\ell_1\|_{L^2} \leq C\|\nabla^{3}n\|_{L^2}\|\nabla^{2} f_1\|_{L^2} \leq C\|\nabla^{3}n\|_{L^2}\|\nabla^{2} (u\cdot \nabla n)\|_{L^2}  \leq C (1+t)^{-10} M(t)^3 .
 	\end{aligned}
 \end{equation*}
This deduces
\begin{equation*}
	\begin{aligned}
		\int_{\mathbb{R}^3} \nabla^{3} f^{h}_1 \cdot \nabla^{3} n^{h} \,dx \leq C (1+t)^{-10} M(t)^3
	\end{aligned}
\end{equation*}
for some $C>0$ independent of time.

For the second term on the right-hand side of \eqref{E21}, we also divide it into two terms:
\begin{align*}
		\int_{\mathbb{R}^3} \nabla^{3} f^{h}_2 \cdot \nabla^{3} u^{h} \,dx= \int_{\mathbb{R}^3} \nabla^{3} f_2 \cdot \nabla^{3} u^{h} \,dx-\int_{\mathbb{R}^3} \nabla^{3} f^{\ell}_2 \cdot \nabla^{3} u^{h} \,dx
		\triangleq J_{3}+J_{4}.
	\end{align*}
Here the first term of $J_3$ can be estimated as	
	\begin{align*}
   &\Big|\int_{\mathbb{R}^3}\nabla^3(u\cdot\nabla u)\cdot\nabla^3u^h\,dx\Big|\\
	&\quad \leq \Big|\int_{\mathbb{R}^3} (\nabla^{3} (u \cdot \nabla u)-u\cdot\nabla^3\nabla u) \cdot \nabla^{3} u^{h} \,dx\Big| +\Big|\int_{\mathbb{R}^3} u\cdot\nabla^3\nabla u \cdot \nabla^{3} u^{h} \,dx\Big| \\
	&\quad \leq \Big|\int_{\mathbb{R}^3} (\nabla^{3} (u \cdot \nabla u)-u\cdot\nabla^3\nabla u) \cdot \nabla^{3} u^{h} \,dx\Big| +\Big|\int_{\mathbb{R}^3} u\cdot\nabla^3\nabla u^h \cdot \nabla^{3} u^{h} \,dx\Big|+\Big|\int_{\mathbb{R}^3} u\cdot\nabla^3\nabla u^\ell \cdot \nabla^{3} u^{h} \,dx\Big|  \\
	&\quad \leq C(\|\nabla u\|_{L^\infty}\|\nabla^3u\|_{L^2}+\|\nabla^3u\|_{L^2}\|\nabla u\|_{L^\infty})\|\nabla^3u^h\|_{L^2}+C\|\text{div}u\|_{L^\infty}\|\nabla^3u^h\|_{L^2}^2\cr
	&\qquad +C\|u\|_{L^\infty}\|\nabla^4u^\ell\|_{L^2}\|\nabla^3u^h\|_{L^2}\\
	&\quad \leq C\|\nabla u\|_{L^\infty}^2\|\nabla^3u\|_{L^2}^2+C\| u\|^{2}_{L^\infty}\|\nabla^{3}u\|^{2}_{L^2}  +\epsilon \|\nabla^{3}u^{h}\|^{2}_{L^2}\\
	&\quad \leq C (1+t)^{-\frac{15}{2}} M(t)^4+\epsilon \|\nabla^{3}u^{h}\|^{2}_{L^2}.
\end{align*}
The other terms of $J_3$ and $J_4$ can be handled in a similar manner. Finally, we obtain 
$$
\int_{\mathbb{R}^3} \nabla^{3} f^{h}_2 \cdot \nabla^{3} u^{h} \,dx=J_{3}+J_4 \leq C(1+t)^{-\frac{15}{2}}(M(t)^3+M(t)^4)+C\epsilon \|\nabla^{3}u^{h}\|^{2}_{L^2}.
$$
We also have 
\begin{align*}
&\int_{\mathbb{R}^3} \nabla^{3} f^{h}_3 \cdot \nabla^{3} v^{h} \,dx+\int_{\mathbb{R}^3} \nabla^{3} (-\Delta)^{-1}n^{h}  \cdot \nabla^{3}f^{h}_1\,dx \cr
&\quad \leq C(1+t)^{-\frac{15}{2}}(M(t)^3+M(t)^4)+C\epsilon \|\nabla^{3}v^{h}\|^{2}_{L^2}.
\end{align*}
The other terms are treated similarly, and thus we have
\begin{equation*}
	\begin{aligned}
		& \frac{1}{2} \frac{d}{d t}\Big\{\|\nabla^{3} (n^{h},u^{h},v^{h})\|_{L^2}^2+\|\nabla^{4}(-\Delta)^{-1}n^{h}\|_{L^2}^{2}\Big\}+\left\|\nabla^4 v^{h}\right\|_{L^2}^2 
		+\|\nabla^{3}(u^{h}-v^{h})\|_{L^2}^2\\
		& \quad \leq C (1+t)^{-\frac{15}{2}} \Big(M(t)^3+M(t)^4\Big)+ C\epsilon (\|\nabla^{3}u^{h}\|^{2}_{L^2}+\|\nabla^{3}v^{h}\|^{2}_{L^2})\\
		& \quad \leq C (1+t)^{-\frac{15}{2}} \Big(M(t)^3+M(t)^4\Big)+ C\epsilon (\|\nabla^{3}(u^{h}-v^h)\|^{2}_{L^2}+\|\nabla^{4}v^{h}\|^{2}_{L^2}).
	\end{aligned}
\end{equation*}
Due to the smallness of  $\epsilon$, we have
\begin{equation}\label{E222}
	\begin{aligned}
		& \frac{1}{2} \frac{d}{d t}\Big\{\|\nabla^{3} (n^{h},u^{h},v^{h})\|_{L^2}^2+\|\nabla^{4}(-\Delta)^{-1}n^{h}\|_{L^2}^{2}\Big\}+\frac{1}{2}\Big\{\left\|\nabla^4 v^{h}\right\|_{L^2}^2 
		+\|\nabla^{3}(u^{h}-v^{h})\|_{L^2}^2\Big\}\\
		& \quad \leq C (1+t)^{-\frac{15}{2}} \Big(M(t)^3+M(t)^4\Big).
	\end{aligned}
\end{equation}

To close the estimate, similarly as in Proposition \ref{p3}, we estimate
\begin{equation*}
	\begin{aligned}
		& \|\nabla^{3}n^{h}\|_{L^2}^2+\int_{\mathbb{R}^3} \nabla^{2} \partial_{t} u^{h} \cdot  \nabla^{3} n^{h} \,dx\\
		& \quad=-\int_{\mathbb{R}^3} \nabla^{2}(u^{h}-v^{h}) \cdot \nabla^{3} n^{h} \,dx-\int_{\mathbb{R}^3} \nabla^{3}(-\Delta)^{-1}n^{h} \cdot \nabla^{3} n^{h} \,dx+\int_{\mathbb{R}^3} \nabla^{2}f^{h}_{2} \cdot \nabla^{3} n^{h} \,dx,
	\end{aligned}
\end{equation*}
then
\begin{equation}\label{E8}
	\begin{aligned}
		& \frac{d}{dt}\int_{\mathbb{R}^3} \nabla^{2}  u^{h} \cdot  \nabla^{3} n^{h} \,dx + \frac{1}{2}\|\nabla^{3}n^{h}\|_{L^2}^2+\|\nabla^4(-\Delta)^{-1}n^h\|_{L^2}^2\\
		&\quad \leq  C (1+t)^{-\frac{15}{2}} M(t)^3+ C\|\nabla^{3}(u^h-v^h)\|_{L^2}^2.
	\end{aligned}
\end{equation}
Let $\beta_2>0$ be a small constant. Taking $\eqref{E222}+\beta_{2} \times \eqref{E8}$ yields that
	\begin{align*}
&\frac{d}{dt}\Big\{\frac{1}{2}\Big(\|\nabla^3(n^h,u^h,v^h)\|_{L^2}^2+\|\nabla^4(-\Delta)^{-1}n^h\|_{L^2}^2\Big)+\beta_2\int_{\mathbb{R}^3}\nabla^{2}  u^{h} \cdot  \nabla^{3} n^{h} \,dx \Big\}\\
&\quad+\frac{1}{2}\Big(\left\|\nabla^4 v^{h}\right\|_{L^2}^2 
+\|\nabla^{3}(u^{h}-v^{h})\|_{L^2}^2\Big)+\beta_2\Big( \frac{1}{2}\|\nabla^{3}n^{h}\|_{L^2}^2+\|\nabla^4(-\Delta)^{-1}n^h\|_{L^2}^2\Big)\\
&\qquad \leq C (1+t)^{-\frac{15}{2}} \Big(M(t)^3+M(t)^4\Big)
+C\beta_2 (1+t)^{-\frac{15}{2}} M(t)^3+ C\beta_2\|\nabla^{3}(u^h-v^h)\|_{L^2}^2.
	\end{align*}
	Choosing $\beta_{2} > 0$ small enough implies
\begin{equation*}
	\begin{aligned}
		& \frac{d}{d t}\Big(\|\nabla^{3} (n^{h},u^{h},v^{h})\|_{L^2}^2+\|\nabla^{4}(-\Delta)^{-1}n^{h}\|_{L^2}^{2}\Big) +C\Big(\|\nabla^{3}(n^{h},u^{h},v^{h})\|_{L^2}^2+\|\nabla^4(-\Delta)^{-1}n^h\|_{L^2}^2\Big)\\
		& \quad \leq C (1+t)^{-\frac{15}{2}} (M(t)^3+M(t)^4).
	\end{aligned}
\end{equation*}
By Gr\"{o}nwall's lemma, it holds 
\begin{equation}\label{032003}
\|\nabla^{3}(n^{h},u^{h},v^{h})\|_{L^2}
\leq C (1+t)^{-\frac{15}{4}} (\delta_{0}+M(t)^{\frac{3}{2}}+M(t)^2) .
\end{equation}
Combining the estimates \eqref{011707}, \eqref{032001}, \eqref{032002}, and \eqref{032003}, we obtain
\begin{equation*}
	\|\nabla^{3}n\|_{L^2}
	\leq	\|\nabla^{3}n^{h}\|_{L^2}+	\|\nabla^{3}n^\ell\|_{L^2} 
	\leq C (1+t)^{-\frac{15}{4}} (\delta_{0}+M^{\frac32}(t)+M(t)^2),
\end{equation*}
\begin{equation*}
	\|\nabla^{3}(u,v)\|_{L^2}
	\leq	\|\nabla^{3}(u^h,v^h)\|_{L^2}+	\|\nabla^{3}(u^\ell,v^\ell)\|_{L^2} 
	\leq C (1+t)^{-\frac{9}{4}} (\delta_{0}+M^{\frac32}(t)+M(t)^2),
\end{equation*}
which, together with the decay rates of $\left\|\nabla^k(n, u, v)(t)\right\|_2(k=0,1,2)$ and the definitions of $M(t)$, gives
\begin{equation*}
	M(t) \leq C (\delta_{0}+M^{\frac32}(t)+M(t)^2) .
\end{equation*}
Due to the smallness of $\delta_0$, we directly observe that
$$
M(t) \leq  C\delta_0.
$$
As a consequence, we deduce
	\begin{equation}\label{032503}
	\begin{aligned}
		\left\|\nabla^{k}n(t)\right\|_{L^2}&\leq C\delta_0(1+t)^{-\frac{11}{4}-\frac{k}{2}}, \quad k=0,1,2, \\
		\left\|\nabla^{3}n(t)\right\|_{L^2}&\leq C \delta_0(1+t)^{-\frac{15}{4}},\\
		\left\|\nabla^{k}(u, v)(t)\right\|_{L^2}&\leq C\delta_0(1+t)^{-\frac{3}{4}-\frac{k}{2}},\quad k=0,1,2,3.
	\end{aligned}
\end{equation}

We next provide the decay rate of $\|u-v\|_{L^2}$. It follows from \eqref{main4} that $u-v$ satisfies
\[
	\partial_t(u-v)+2(u-v)=-\nabla n-\nabla(-\Delta)^{-1}n+v\cdot\nabla v+\nabla P-\Delta v+f_2-(e^n-1)(u-v). 
\]
The critical aspect lies in estimating $\|\nabla(-\Delta)^{-1}n\|_{L^2}$. First, we deduce
	\begin{align*}	
		&\left\|\nabla (-\Delta)^{-1}n(t)\right\|_{L^2} \cr
		&\quad \leq\left\|\nabla (-\Delta)^{-1} S_{11} * n_0\right\|_{L^2}+\left\|\nabla (-\Delta)^{-1}S_{12} * u_0\right\|_{L^2}\\
		&\qquad +\int_0^t\left\|\nabla (-\Delta)^{-1}S_{11}(t-s) * f_1(s)\right\|_{L^2} \,ds+\int_0^t\left\|\nabla (-\Delta)^{-1}S_{12}(t-s) * f_2(s)\right\|_{L^2} \,ds \\
		&\quad \leq Ce^{-R t}\left(\left\|\nabla(-\Delta)^{-1} n_0\right\|_{L^2}+\left\|\nabla(-\Delta)^{-1} u_0\right\|_{L^2}\right)\\
		&\qquad +C\int_0^t e^{-R(t-s)}\Big(\|\nabla(-\Delta)^{-1}f_1(s)\|_{L^2}+\|\nabla(-\Delta)^{-1}f_2(s)\|_{L^2}\Big) \,ds \\
		&\quad \leq Ce^{-Rt}\left(\left\| n_0\right\|_{\dot{H}^{-1}}+\|u_0\|_{L^2}+\|u_0\|_{L^1}\right)\\
		&\qquad +C\int_0^te^{-R(t-s)}(\|f_1(s)\|_{L^2}+\|f_1(s)\|_{L^1}+\|f_2(s)\|_{L^2}+\|f_2(s)\|_{L^1})\,ds\\
		&\quad \leq Ce^{-Rt}\delta_0+C\int_0^t e^{-R(t-s)}\Big\{\left\|(u \cdot \nabla n)\right\|_{L^2}+\left\|(u \cdot \nabla n)\right\|_{L^1}+\left\|(u \cdot \nabla u)\right\|_{L^2}+\left\|(u \cdot \nabla u)\right\|_{L^1}\\
		&\hspace{4cm} + \left\|\nabla (-\Delta)^{-1}(e^n-1-n) \right\|_{L^2}+\left\|\nabla (-\Delta)^{-1}(e^n-1-n) \right\|_{L^1}\Big\}\,ds.
	\end{align*}
By \eqref{032503}, we show that 
	\begin{align*}	
	\left\|\nabla (-\Delta)^{-1}n(t)\right\|_{L^2}  \leq Ce^{-Rt}\delta_0+C\delta_0^2\int_0^t e^{-R(t-s)}(1+s)^{-2}\,ds \leq C\delta_0(1+t)^{-2}.
\end{align*}
We then estimate $L^2$ norm of $u-v$ to get
$$
\begin{aligned}
&\frac{1}{2}\frac{d}{dt}\|u-v\|_{L^2}^2+2\|u-v\|_{L^2}^2\\
&\quad \leq \left| \int_{\mathbb{R}^3}(-\nabla n-\nabla(-\Delta)^{-1}n+v\cdot\nabla v+\nabla P-\Delta v+f_2-(e^n-1)(u-v))\cdot (u-v)\,dx\right|\\
&\quad \leq C_\epsilon(\|\nabla n\|_{L^2}^2+\|\nabla(-\Delta)^{-1}n\|_{L^2}^2+\|v\cdot\nabla v\|_{L^2}^2+\|\nabla P\|_{L^2}^2+\|\nabla^2v\|_{L^2}^2+\|f_2\|_{L^2}^2)\\
&\qquad +\epsilon\|u-v\|_{L^2}^2+\|e^n-1\|_{L^\infty}\|u-v\|_{L^2}^2\\
&\quad \leq C\delta_0^2((1+t)^{-\frac{13}{2}}+(1+t)^{-4}+(1+t)^{-\frac{11}{2}}+(1+t)^{-\frac{7}{2}})+2\epsilon\|u-v\|_{L^2}^2+\frac{5}{4}\|u-v\|_{L^2}^2\\
&\quad \leq C\delta_0^2(1+t)^{-\frac{7}{2}}+(\frac{5}{4}+2\epsilon)\|u-v\|_{L^2}^2\\
&\quad \leq C\delta_0^2(1+t)^{-\frac{7}{2}}+\frac{3}{2}\|u-v\|_{L^2}^2,
\end{aligned}
$$
where we used the smallness of $\epsilon$ and
\begin{align*}
\|\nabla P\|_{L^2} &\leq \|v\cdot\nabla v\|_{L^2}+\|\rho(u-v)\|_{L^2} \cr
&\leq \|v\cdot\nabla v\|_{L^2}+\|\rho\|_{L^\infty}\|(u-v)\|_{L^2} \cr
&\leq C\delta_0(1+t)^{-\frac{11}{4}}+\frac{5}{4}\|u-v\|_{L^2}.
\end{align*}
Solving this inequality yields
\begin{equation*}
	\|(u-v)(t)\|_{L^2}\leq C\delta_0(1+t)^{-\frac{7}{4}}.
\end{equation*}
Similarly, it also holds that
 \begin{equation*}
 	\|\nabla(u-v)(t)\|_{L^2}\leq C\delta_0(1+t)^{-\frac{9}{4}}.
 \end{equation*}
This completes the proof.
\end{proof}
%
%
%
%
%
%
\section{Optimal decay temporal estimates for the EP-NS system}\label{Sec5}
In this section, we provide the lower bound estimates of decay rates of solutions to the EP-NS system thereby completing the proof of  Theorem \ref{Th3}.

We start by presenting the classical interpolation inequality in the lemma below.
\begin{lem}\label{lema5}
	Let $a\geq 0$ and integer $l \geq 0$, then we have
	\begin{equation*}
		\|\nabla^l f\|_{L^2}\lesssim \|\nabla^{l+1}f\|_{L^2}^{1-\theta}\|\Lambda^{-a}f\|_{L^2}^\theta\quad \text{with}\quad \theta=\frac{1}{1+l+a}.
	\end{equation*}
\end{lem}
\begin{proof}		
By Parseval's identity and H\"{o}lder's inequality, we easily find
	\begin{equation*}
		\|\nabla^l f\|_{L^2}=\Big\|(i\xi)^l\hat{f}\Big\|_{L^2}
		\lesssim\big\|(i\xi)^{l+1}\hat{f}\big\|_{L^2}^{1-\theta}\big\||\xi|^{-a}\hat{f}\big\|_{L^2}^\theta=\|\nabla^{l+1}f\|_{L^2}^{1-\theta}\|\Lambda^{-a}f\|_{L^2}^\theta,
	\end{equation*}
	for $\theta=\frac{1}{1+l+a}$. This completes the proof.
\end{proof}

%
%
%
%
%
%
\subsection{Proof of Theorem \ref{thm3}}

We first divide the velocity $u$ into two terms:
		\begin{align*}
				u(t)&= \Big( S_{21} * n_0+S_{22} * u_0+S_{23} * v_0\Big) \\
				&\quad +\Big(\int_0^t S_{21}(t-s) * f_1(s) \,ds+\int_0^t S_{22}(t-s)* f_2(s) \,ds+\int_0^t S_{23}(t-s) * f_3(s) \,ds\Big)\\
				&\triangleq N_1+N_2.
			\end{align*}
We then estimate
\[
		\begin{aligned}
			\|N_2\|_{L^2}&\leq \int_0^t \|S_{21}(t-s) * f_1(s)\|_{L^2} \,ds+\int_0^t \|S_{22}(t-s)* f_2(s)\|_{L^2}\,ds  +\int_0^t \|S_{23}(t-s) * f_3(s) \|_{L^2}\,ds\\
			&\leq C\int_0^te^{-R(t-s)}(\|f_1\|_{L^2}+\|f_1\|_{L^1})\,ds
			+C\int_0^t(1+t-s)^{-\frac{3}{4}}(\|h_2\|_{L^1}+\|h_2\|_{L^2})\,ds\\
			&\quad+C\int_0^t(1+t-s)^{-\frac{3}{4}}(\|h_3\|_{L^1}+\|h_3\|_{L^2})\,ds,
		\end{aligned}
\]
		where $f_i$, $i=1,2,3$ and $h_j$, $j=2,3$ are given as in \eqref{nonlinearterms} and \eqref{h1}, respectively.
		
On the other hand, using \eqref{lemA1-1}, we obtain
\[
			\begin{aligned}
				 \|f_1\|_{L^1}&\leq \|u\|_{L^2}\|\nabla n\|_{L^2}\leq C\delta_0^2(1+t)^{-4},\\
		 \|f_1\|_{L^2}&\leq \|u\|_{L^\infty}\|\nabla n\|_{L^2}\leq C\delta_0^2(1+t)^{-\frac{19}{4}},\\
		 \|h_2\|_{L^1}&\leq \|u\|_{L^2}\|\nabla u\|_{L^2}\leq C\delta_0^2(1+t)^{-2},\\
		  \|h_2\|_{L^2}&\leq \|u\|_{L^\infty}\|\nabla u\|_{L^2}\leq C\delta_0^2(1+t)^{-\frac{11}{4}},\\
		 \|h_3\|_{L^1}&\leq \|v\|_{L^2}\|\nabla v\|_{L^2}+\|e^n-1\|_{L^2}\|u-v\|_{L^2}\leq C\delta_0^2(1+t)^{-2},\\
		 \|h_3\|_{L^2}&\leq \|v\|_{L^\infty}\|\nabla v\|_{L^2}+\|e^n-1\|_{L^2}\|u-v\|_{L^\infty}\leq C\delta_0^2(1+t)^{-\frac{11}{4}}.
		 \end{aligned}
\]
Thus, we deduce
			\begin{equation*}
			\begin{aligned}
				\|N_2\|_{L^2}
				&\leq C\delta_0^2\int_0^te^{-R(t-s)}(1+s)^{-\frac{11}{4}}\,ds +C\delta_0^2\int_0^t(1+t-s)^{-\frac{3}{4}}(1+s)^{-2}\,ds\\
				&\quad +C\delta_0^2\int_0^t(1+t-s)^{-\frac{3}{4}}(1+s)^{-2}\,ds\\
				&\leq C\delta_0^2(1+t)^{-\frac{3}{4}},
			\end{aligned}
		\end{equation*}
		and subsequently, by Parseval's identity, Proposition \ref{pro-li-d}, and the smallness of $\delta_0$, we have
		\begin{equation}\label{thm3-p1}
			\begin{aligned}
				\|u(t)\|_{L^2}
				&\geq \|N_1\|_{L^2}-\|N_2\|_{L^2} \geq d_*(1+t)^{-\frac{3}{4}}- C\delta_0^2(1+t)^{-\frac{3}{4}} \geq \frac{1}{2}d_*(1+t)^{-\frac{3}{4}}.
			\end{aligned}
		\end{equation}
		
For the lower bound estimates for the derivatives, we estimate 		
		\begin{equation}\label{thm3-p4}
			\begin{aligned}
				&\|\Lambda^{-1}u(t)\|_{L^2}\cr
				&\quad \leq \|\Lambda^{-1}u^\ell(t)\|_{L^2}
				+\|\Lambda^{-1}u^h(t)\|_{L^2}\\
				&\quad \leq C(1+t)^{-\frac{1}{4}}+C\int_0^t(1+t-s)^{-\frac{1}{4}}(\|f_1(s)\|_{L^1} +\|f_1(s)\|_{L^2} +\|h_2(s)\|_{L^1}+\|h_3(s)\|_{L^1})\,ds\cr
				&\qquad +C\|u^h(t)\|_{L^2}\\
				&\quad \leq C(1+t)^{-\frac{1}{4}}+C\int_0^t(1+t-s)^{-\frac{1}{4}}(1+s)^{-2}\,d\tau+C\|u(t)\|_{L^2}\\
				&\quad \leq C(1+t)^{-\frac{1}{4}}
			\end{aligned}
		\end{equation}
		for $t>0$ large enough. Then, it follows from Lemma \ref{lema5} that 
		\begin{equation*}
			\|u\|_{L^2}\leq C\|\Lambda^{-1}u\|_{L^2}^{\frac{k}{k+1}}
			\|\nabla^k u\|_{L^2}^{\frac{1}{k+1}},
		\end{equation*}
		which, together with \eqref{thm3-p1} and \eqref{thm3-p4}, yields for $k=0,1,2,3$ that 
		\begin{equation*}
		\|\nabla^ku(t)\|_{L^2}\geq c_*(1+t)^{-\frac{3}{4}-\frac{k}{2}},
		\end{equation*}
		with $c_*$ a positive constant. Similarly, we can also obtain the lower bound of the decay rates of velocity $v$
		\begin{equation*}
		\|\nabla^kv(t)\|_{L^2}\geq c_*(1+t)^{-\frac{3}{4}-\frac{k}{2}}.
		\end{equation*}
			
To obtain the lower bound of the decay rates for the difference $u-v$, we first write 
\begin{align*}
	u-v&=
	\Big\{ S_{21} * n_0+S_{22} * u_0+S_{23} * v_0-( S_{32} * u_0+S_{33} * v_0)\Big\} \\
	&\quad +\Big\{\int_0^t S_{21}(t-s) * f_1(s) \,ds+\int_0^t (S_{22}-S_{32})(t-s)* f_2(s) \,ds+\int_0^t (S_{23}-S_{33})(t-s)* f_3(s) \,ds\Big\}\\
	&\triangleq N_3+N_4.
\end{align*}
Here we notice that  $N_3$ is the linear part of $u-v$, and thus by Proposition \ref{pro-li-d}, we get
\begin{equation*}
	\bar{d}_*(1+t)^{-\frac{7}{4}}\leq \|N_3\|_{L^2}\leq C(1+t)^{-\frac{7}{4}}.
\end{equation*}
For the nonlinear term $N_4$, by Proposition \ref{lemA1}, the first term can be estimated as 
\begin{align*}
\left\|\int_0^t S_{21}(t-s) * f_1(s) \,ds\right\|_{L^2}&\leq C\int_0^te^{-R(t-s)}(\|f_1(s)\|_{L^1}+\|f_1(s)\|_{L^2})\,ds\\
&\leq C\delta_0^2\int_0^te^{-R(t-s)}\Big((1+s)^{-4}+(1+s)^{-\frac{19}{4}})\,ds\\
&\leq C\delta_0^2(1+t)^{-4}.
\end{align*}
For the second term of $N_4$, we obtain 
\begin{align*}
&\left\|\int_0^t (S_{22}-S_{32})(t-s)* f_2(s) \,ds\right\|_{L^2}\\
&\quad \leq \left\|\int_0^t\Big\{\frac{(\lambda_3+2)e^{\lambda_4t}-(\lambda_4+2)e^{\lambda_3t}}{\lambda_3-\lambda_4}\big(I-\frac{\xi\xi^t}{|\xi|^2}\big)+\frac{\lambda_1e^{\lambda_1t}-\lambda_2e^{\lambda_2t}}{\lambda_1-\lambda_2}\frac{\xi\xi^t}{|\xi|^2}\Big\}\hat{f}_2(s)\,ds\right\|_{L^2}.
\end{align*}
By Proposition  \ref{pro-li-d}, it holds for $|\xi|\leq r_0$ that 
\begin{align*}
	\frac{(\lambda_3+2)e^{\lambda_4t}-(\lambda_4+2)e^{\lambda_3t}}{\lambda_3-\lambda_4}\big(I-\frac{\xi\xi^t}{|\xi|^2}\big)\simeq \frac{1}{4}|\xi|^2e^{-\frac{1}{2}|\xi|^2t}\big(I-\frac{\xi\xi^t}{|\xi|^2}\big),
\end{align*}
and for $|\xi|>r_0$
\begin{align*}
	\frac{(\lambda_3+2)e^{\lambda_4t}-(\lambda_4+2)e^{\lambda_3t}}{\lambda_3-\lambda_4}\big(I-\frac{\xi\xi^t}{|\xi|^2}\big)\simeq e^{-Rt}.
\end{align*}
Thus we have 
\begin{align*}
	&\left\|\int_0^t (S_{22}-S_{32})(t-s)* f_2(s) \,ds\right\|_{L^2}\\
	&\quad \leq \left\|\int_0^t\Big\{\frac{(\lambda_3+2)e^{\lambda_4t}-(\lambda_4+2)e^{\lambda_3t}}{\lambda_3-\lambda_4}\big(I-\frac{\xi\xi^t}{|\xi|^2}\big)+\frac{\lambda_1e^{\lambda_1t}-\lambda_2e^{\lambda_2t}}{\lambda_1-\lambda_2}\frac{\xi\xi^t}{|\xi|^2}\Big\}\hat{f}_2(s)\,ds\right\|_{L^2}\\
	&\quad \leq C\int_0^t \left\|\frac{1}{4}|\xi|^2e^{-\frac{1}{2}|\xi|^2t}\big(I-\frac{\xi\xi^t}{|\xi|^2}\big)\right\|_{L^2}\|f_2(s)\|_{L^1}\,ds+C\int_0^te^{-R(t-s)}\|f_2(s)\|_{L^2}\,ds\\
	&\quad \leq C\int_0^t(1+t-s)^{-\frac{7}{4}}(\|f_2(s)\|_{L^1}+\|f_2(s)\|_{L^2})\,ds\\
	&\quad \leq C\delta_0^2\int_0^t(1+t-s)^{-\frac{7}{4}}\Big((1+s)^{-2}+(1+s)^{-\frac{11}{4}}\Big)\,ds\\
	&\quad \leq C\delta_0^2(1+t)^{-\frac{7}{4}}.
\end{align*}
Similarly, we also get
\begin{align*}
	\left\|\int_0^t (S_{23}-S_{33})(t-s) * f_3(s) \,ds\right\|_{L^2}\leq  C\delta_0^2(1+t)^{-\frac{7}{4}}.
\end{align*}
Combining the above estimates, we arrive at
\begin{align*}
	\|N_4\|_{L^2}\leq C\delta_0^2(1+t)^{-\frac{7}{4}},
\end{align*}
and subsequently, for large time, we conclude
\begin{align}\label{051401}
	\|u-v\|_{L^2}\geq \|N_3\|_{L^2}-\|N_4\|_{L^2}\geq \bar{d}_*(1+t)^{-\frac{7}{4}}-C\delta_0^2(1+t)^{-\frac{7}{4}}\geq \frac{1}{2}\bar{d}_*(1+t)^{-\frac{7}{4}}.
\end{align}
We take a similar way as \eqref{thm3-p4}, which gives 
\begin{align}\label{051402}
\|\Lambda^{-1}(u-v)\|_{L^2}\leq C(1+t)^{-\frac{5}{4}}.
\end{align}
Then by Lemma \ref{lema5} and \eqref{051401}-\eqref{051402}, there exists a positive constant $\bar{c}_*$ such that 
	\begin{align*}
	\|\nabla^k(u-v)\|_{L^2}&\geq C \|\Lambda^{-1}(u-v)\|_{L^2}^{-k}
	\|u-v\|_{L^2}^{k+1}\geq \bar{c}_*(1+t)^{-\frac{7}{4}-\frac{k}{2}},\quad k=0,1.
\end{align*}
Combining \eqref{040101} in Theorem \ref{thm1}, we complete the proof of Theorem \ref{thm3}.

%
%
%
%
%
%

\section{Exponential decay estimates for  EP system with damping}\label{Sec6}

As observed in Section \ref{ssec_lina}, the linear analysis demonstrates exponential decay of solutions to the equilibrium when there is no coupling with incompressible flow. Thus, in this section, we consider the following the damped Euler-Poisson system:
\[
	\left\{
	\begin{aligned}
		&\partial_t n+u\cdot\nabla n +\text{div}u=0,\\
		&\partial_tu+u\cdot\nabla u+\nabla n+u=-\nabla U, \\
		&-\Delta U=e^n-1,
	\end{aligned}
	\right.
\]
and show the global-in-time existence and uniqueness of classical solutions and their exponential decay rate of convergences towards the equilibrium, providing a detailed proof of Theorem \ref{Th3}.
%
%
%
%
%
%
\subsection{Proof of Theorem \ref{Th3}}

The proof of local existence theory for Euler-Poisson with damping resembles that of Theorem \ref{lem-loc}. For the global-in-time regularity of solutions, we need the uniform-in-time bound estimates of the solutions in the desired solution space. To this end, we define
\begin{equation*}
	\mathcal{A}(t)=\|(n, u )(t)\|_{H^3}+\|n(t)\|_{\dot{H}^{-1}} \quad \mbox{and} \quad \mathcal{B}(t)=\|n(t)\|_{H^3}+\|u(t)\|_{H^3},
\end{equation*}
and assume the following a priori assumption on the solutions:
\begin{equation*}
	\mathcal{A}(t)=\|(n,v)(t)\|_{H^3}+\|n(t)\|_{\dot{H}^{-1}} \leq \delta_*,
\end{equation*}
with $\delta_*$ a sufficiently small positive constant.	

Following a similar process as in Propositions \ref{p1}, \ref{p2}, and \ref{p3}, we obtain that there exists a constant $\bar{C}_0>0$ such that 
\begin{equation}\label{032506}
\frac{d}{d t}\left\{\|(n, u)\|_{L^2}^2+\left\|\nabla (-\Delta)^{-1}n\right\|_{L^2}^2\right\}+\bar{C}_0\|u\|_{L^2}^2\leq 0,
\end{equation}
and for $k=0,1,2$ that
\begin{equation*}
	\begin{aligned}
		& \frac{1}{2} \frac{d}{d t}\left(\left\|\nabla^k n\right\|_{H^1}^2+\left\|\nabla^{k+1} u\right\|_{L^2}^2\right)+\left\|\nabla^{k+1}u\right\|_{L^2}^2
		 \leq C\mathcal{A}(t)\mathcal{B}^2(t).
	\end{aligned}
\end{equation*}
Summing these inequality over $k$ form 0 to 2 and combining with \eqref{032506}, we have 
\begin{equation}\label{032509}
	\begin{aligned}
		& \frac{1}{2} \frac{d}{d t}\left(\left\| n\right\|_{H^3}^2+\left\|\nabla(-\Delta)^{-1}n\right\|_{L^2}^2+\left\| u\right\|_{H^3}^2\right)+(1+\bar{C}_0)\left\|u\right\|_{H^3}^2
		\leq C\mathcal{A}(t)\mathcal{B}^2(t).
	\end{aligned}
\end{equation}
Parallel to Proposition \ref{p3}, we find
\begin{equation}\label{032508}
	\begin{aligned}
		 \frac{d}{d t}\left\{\sum_{k= 0}^2 \int_{\mathbb{R}^3} \nabla^{k} u \cdot \nabla^{k+1} n \,dx\right\}+\frac{1}{2}\|n\|_{H^3}^2 
		& \leq \bar{C}_1\left(\|\nabla u\|_{H^2}^2+\|u\|_{H^2}^2\right)+C\mathcal{A}(t)\mathcal{B}^2(t)\\
		&\leq  \bar{C}_1\|u\|_{H^3}^2+C\mathcal{A}(t)\mathcal{B}^2(t)
	\end{aligned}
\end{equation}
for some $\bar{C}_1 > 0$ independent of time.

To complete our analysis, we need to establish dissipation estimates for $\|\nabla (-\Delta)^{-1}n\|_{L^2}^2$. We consider the equation:
$$
\partial_tu+\nabla n+u+\nabla (-\Delta )^{-1}n=f_2.
$$
Taking the $L^2$ inner producet with $\nabla (-\Delta )^{-1}n$ in $\mathbb{R}^3$, we find
$$
\begin{aligned}
&\frac{d}{dt}\int_{\mathbb{R}^3}u\cdot\nabla(-\Delta)^{-1}n\,dx+\|\nabla (-\Delta )^{-1}n\|_{L^2}^2+\|n\|_{L^2}^2\\
&\quad =-\int_{\mathbb{R}^3}u\cdot\nabla (-\Delta )^{-1}n\,dx+\int_{\mathbb{R}^3}u\cdot \nabla (-\Delta )^{-1}\partial_tn\,dx+\int_{\mathbb{R}^3}\nabla(-\Delta)^{-1}n\cdot f_2\,dx\\
&\quad =-\int_{\mathbb{R}^3}u\cdot\nabla (-\Delta )^{-1}n\,dx+\int_{\mathbb{R}^3}u\cdot \nabla (-\Delta )^{-1}(-\text{div}u+f_1)\,dx+\int_{\mathbb{R}^3}\nabla(-\Delta)^{-1}n\cdot f_2\,dx\\
&\quad \leq \bar{C}_2 \|u\|_{L^2}^2+\frac{1}{2}\|\nabla (-\Delta )^{-1}n\|_{L^2}^2+
C\mathcal{A}(t)\mathcal{B}^2(t),
\end{aligned}
$$
where $\bar{C}_2>0$ is a positive constant independent of time. This leads to
\begin{equation}\label{032505}
\frac{d}{dt}\int_{\mathbb{R}^3}u\cdot\nabla(-\Delta)^{-1}n\,dx+\frac{1}{2}\|\nabla (-\Delta )^{-1}n\|_{L^2}^2+\|n\|_{L^2}^2\leq 
\bar{C}_2\|u\|_{L^2}^2+C\mathcal{A}(t)\mathcal{B}^2(t).
\end{equation}
Taking $\eqref{032509}+\beta_3\times \eqref{032508}+\beta_3\times\eqref{032505}$ with $\beta_3 > 0$  small enough, we obtain
\begin{equation*}
	\begin{aligned}
		& \frac{d}{d t}\left\{ \frac{1}{2}\Big(\left\| n\right\|_{H^3}^2+\left\|\nabla(-\Delta)^{-1}n\right\|_{L^2}^2+\left\| u\right\|_{H^3}^2\Big)+\beta_3 \sum_{k= 0}^2 \int_{\mathbb{R}^3} \nabla^{k} u \cdot \nabla^{k+1} n \,dx
		+\beta_3 \int_{\mathbb{R}^3}u\cdot\nabla(-\Delta)^{-1}n\,dx
		\right\}\\
		& \quad +(1+\bar{C}_0)\left\|u\right\|_{H^3}^2+\frac{1}{2}\beta_3\|n\|_{H^3}^2 +\frac{1}{2}\beta_3\|\nabla (-\Delta )^{-1}n\|_{L^2}^2+\beta_3\|n\|_{L^2}^2\\
		&\qquad \leq (\bar{C}_1+\bar{C}_2)\beta_3\|u\|_{H^3}^2+
		C\mathcal{A}(t)\mathcal{B}^2(t)+C\beta_3\mathcal{A}(t)\mathcal{B}^2(t).
	\end{aligned}
\end{equation*}
Due to the smallness of $\beta_3$, we find a positive constant $\bar{C}_3$ such that 
\begin{equation*}
	\begin{aligned}
		& \frac{d}{d t}\left\{\left\| n\right\|_{H^3}^2+\left\|n\right\|_{\dot{H}^{-1}}^2+\left\| u\right\|_{H^3}^2\right\} +\bar{C}_3\left\{\left\| n\right\|_{H^3}^2+\left\|n\right\|_{\dot{H}^{-1}}^2+\left\| u\right\|_{H^3}^2\right\}\\
		&\quad \leq C\mathcal{A}(t)\mathcal{B}^2(t)\leq C\delta_*\mathcal{B}^2(t).
	\end{aligned}
\end{equation*}
According to the definition $\mathcal{B}(t)$ and the smallness of $\delta_*$, we arrive at
\begin{equation}\label{032511}
		\frac{d}{d t}\left\{\left\| n\right\|_{H^3}^2+\left\|n\right\|_{\dot{H}^{-1}}^2+\left\| u\right\|_{H^3}^2\right\}+\frac{1}{2}\bar{C}_3\left\{\left\| n\right\|_{H^3}^2+\left\|n\right\|_{\dot{H}^{-1}}^2+\left\| u\right\|_{H^3}^2\right\}\leq 0.
\end{equation}
Integrating the above inequality with respect to time from $0$ to $t$ yields 
\begin{align*}
\left\| n\right\|_{H^3}^2+\left\|n\right\|_{\dot{H}^{-1}}^2+\left\| u\right\|_{H^3}^2+\frac{1}{2}\bar{C}_3\int_0^t\left\{\left\| n(s)\right\|_{H^3}^2+\left\|n(s)\right\|_{\dot{H}^{-1}}^2+\left\| u(s)\right\|_{H^3}^2\right\}ds \leq C\delta_0^2,
\end{align*}
which guarantees the uniform-in-time boundedness of solutions. Then by a continuation argument, we establish the global-in-time well-posedness.

Moreover, applying Gr\"{o}nwall's lemma to \eqref{032511}, we have
\begin{equation*}
\left\| n(t)\right\|_{H^3}+\left\|n(t)\right\|_{\dot{H}^{-1}}+\left\| u(t)\right\|_{H^3}\leq C\delta_0e^{-\frac{1}{2}\bar{C}_3 t}.
\end{equation*} 
Finally, we set $C_*=\frac{1}{2}\bar{C}_3$ to complete the proof of Theorem \ref{Th3}.

	\vskip 5mm

%
%
%
%
%
%
\subsubsection*{Acknowledgments} The work of the first author was supported by National Research Foundation of Korea (NRF) grant funded by the Korea government (MSIP) (No. 2022R1A2C1002820). The work of the second author was supported by China Postdoctoral Science Foundation (Grant No. 2023M733691). The work of the third author was supported by  National Natural Science Foundation of China (Grant No. 12001033).
%
%
%
%
%
%


\begin{thebibliography}{100}
					\bibitem{MR2226800}
					C.~Baranger, L.~Boudin, P.-E. Jabin, and S.~Mancini.
					\newblock A modeling of biospray for the upper airways.
					\newblock In {\em C{EMRACS} 2004---mathematics and applications to biology and
						medicine}, volume~14 of {\em ESAIM Proc.}, pages 41--47. EDP Sci., Les Ulis,
					2005.
					
					\bibitem{MR3396231}
					L.~Boudin, C.~Grandmont, A.~Lorz, and A.~Moussa.
					\newblock Modelling and numerics for respiratory aerosols.
					\newblock {\em Commun. Comput. Phys.}, 18(3):723--756, 2015.
					
					\bibitem{CCK16}
					J. A. Carrillo, Y.-P. Choi, and T. K. Karper. 
					\newblock On the analysis of a coupled kinetic-fluid model with local alignment forces.
					\newblock {\em Ann. Inst. H. Poincar\'e Anal. Non Lin\'eaire}, 33(2):273--307, 2016.
					
					\bibitem{MR3564592}
					J.~A. Carrillo, Y.-P. Choi, and E.~Zatorska.
					\newblock On the pressureless damped {E}uler-{P}oisson equations with quadratic
					confinement: critical thresholds and large-time behavior.
					\newblock {\em Math. Models Methods Appl. Sci.}, 26(12):2311--2340, 2016.
					
					\bibitem{MR1400743}
					G.-Q. Chen and D.~Wang.
					\newblock Convergence of shock capturing schemes for the compressible
					{E}uler-{P}oisson equations.
					\newblock {\em Comm. Math. Phys.}, 179(2):333--364, 1996.
					
					\bibitem{MR3546341}
					Y.-P. Choi.
					\newblock Global classical solutions and large-time behavior of the two-phase
					fluid model.
					\newblock {\em SIAM J. Math. Anal.}, 48(5):3090--3122, 2016.
					
					\bibitem{MR4604417}
					Y.-P. Choi and J.~Jung.
					\newblock On the dynamics of charged particles in an incompressible flow: from
					kinetic-fluid to fluid-fluid models.
					\newblock {\em Commun. Contemp. Math.}, 25(7):Paper No. 2250012, 78, 2023.
					
					\bibitem{CJ21}					
					Y.-P. Choi and J. Jung. 
					\newblock On the large-time behavior of Euler-Poisson/Navier-Stokes equations.
					\newblock {\em Appl. Math. Lett.}, 118:107123, 2021.
					
					
					\bibitem{CJK24} 
					Y.-P. Choi, J. Jung and J. Kim.
					\newblock A revisit to the pressureless Euler-Navie-Stokes system in the whole space and its optimal temporal decay. 
					\newblock {\em J. Differential Equations}, 401:231--281, 2024.
					
					\bibitem{MR2601394}
					L.~Desvillettes.
					\newblock Some aspects of the modeling at different scales of multiphase flows.
					\newblock {\em Comput. Methods Appl. Mech. Engrg.}, 199(21-22):1265--1267,
					2010.
					
					\bibitem{MR1416291}
					S.~Engelberg.
					\newblock Formation of singularities in the {E}uler and {E}uler-{P}oisson
					equations.
					\newblock {\em Phys. D}, 98(1):67--74, 1996.
					
					\bibitem{MR1855666}
					S.~Engelberg, H.~Liu, and E.~Tadmor.
					\newblock Critical thresholds in {E}uler-{P}oisson equations.
					\newblock {\em Indiana Univ. Math. J.}, 50:109--157, 2001.
					
					\bibitem{MR1637856}
					Y.~Guo.
					\newblock Smooth irrotational flows in the large to the {E}uler-{P}oisson
					system in {$\mathbb{R}^{3+1}$}.
					\newblock {\em Comm. Math. Phys.}, 195(2):249--265, 1998.
					
					\bibitem{GHZ17}
					Y. Guo, L. Han, and J. Zhang.
					\newblock Absence of shocks for one dimensional {E}uler-{P}oisson system.
					\newblock {\em Arch. Ration. Mech. Anal.}, 223:1057--1121, 2017.
					
					\bibitem{huang2024global}
					F.~Huang, H.~Tang, G.~Wu, and W.~Zou.
					\newblock Global well-posedness and large-time behavior of classical solutions
					to the {E}uler-{N}avier-{S}tokes system in $\mathbb{R}^3$.
					\newblock preprint.
					
					\bibitem{IP13}
					A. D. Ionescu and B. Pausader.
					\newblock The Euler-Poisson system in 2D: global stability of the constant equilibrium solution.
					\newblock {\em Int. Math. Res. Not.}, 2013(4):761--826, 2013.
					
					\bibitem{MR2609958}
					H.-L. Li, A.~Matsumura, and G.~Zhang.
					\newblock Optimal decay rate of the compressible {N}avier-{S}tokes-{P}oisson
					system in {$\mathbb{R}^3$}.
					\newblock {\em Arch. Ration. Mech. Anal.}, 196(2):681--713, 2010.
					
					\bibitem{LW14}
					D. Li and Y. Wu.
					\newblock The Cauchy problem for the two dimensional Euler-Poisson system.
					\newblock {\em J. Eur. Math. Soc.}, 16(10):2211--2266, 2014.
								
					\bibitem{MR1918784}
					H.~Liu and E.~Tadmor.
					\newblock Spectral dynamics of the velocity gradient field in restricted flows.
					\newblock {\em Comm. Math. Phys.}, 228(3):435--466, 2002.
					
					\bibitem{o1981collective}
					P.~J. O'Rourke.
					\newblock Collective drop effects on vaporizing liquid sprays.
					\newblock {\em PhD thesis}, Princeton University, 1981.
					
					\bibitem{MR0837929}
					M.~E. Schonbek.
					\newblock Large time behaviour of solutions to the {N}avier-{S}tokes equations.
					\newblock {\em Comm. Partial Differential Equations}, 11(7):733-763, 1986.
					
					\bibitem{MR2917409}
					Y.~Wang.
					\newblock Decay of the {N}avier-{S}tokes-{P}oisson equations.
					\newblock {\em J. Differential Equations}, 253(1):273--297, 2012.
					
					\bibitem{Williams1958Spray}
					F. A. Williams.
					\newblock Spray combustion and atomization.
					\newblock {\em Physics of Fluids}, 1(6):541, 1958.
					
					\bibitem{MR4175837}
					G.~Wu, Y.~Zhang, and L.~Zhou.
					\newblock Optimal large-time behavior of the two-phase fluid model in the whole
					space.
					\newblock {\em SIAM J. Math. Anal.}, 52(6):5748--5774, 2020.
					
					\bibitem{MR4703478}
					Z.~Wu and W.~Zhou.
					\newblock Pointwise space-time estimates of two-phase fluid model in dimension
					three.
					\newblock {\em J. Evol. Equ.}, 24(1):Paper No. 11, 28, 2024.
					
				\end{thebibliography}
			\end{document}